\documentclass[11pt]{article}
\usepackage[T1]{fontenc}
\usepackage[utf8]{inputenc}

\usepackage{hyperref}
\hypersetup{
  colorlinks=true,
  linkcolor=black,
  citecolor=black,
  urlcolor=blue,
  pdftitle={Radial solutions to a chemotaxis-consumption model involving prescribed signal concentrations on the boundary},
  pdfauthor={Lankeit, Winkler},
  bookmarksopen=true,
}

\usepackage{amsmath,amsthm}
\usepackage{amsfonts}
\usepackage{color}
\usepackage{amsmath,amssymb,comment}

\newtheorem{theo}{Theorem}[section]
\newtheorem{lem}[theo]{Lemma}
\newtheorem{cor}[theo]{Corollary}

\newtheorem{defi}[theo]{Definition}

\newcommand{\mysection}[1]{\section{#1} \setcounter{equation}{0}}

\newcommand{\be}{\begin{equation} \label}
\newcommand{\ee}{\end{equation}}
\newcommand{\bea}{\begin{eqnarray}\label}
\newcommand{\eea}{\end{eqnarray}}
\newcommand{\bas}{\begin{eqnarray*}}
\newcommand{\eas}{\end{eqnarray*}}
\newcommand{\bit}{\begin{itemize}}
\newcommand{\eit}{\end{itemize}}
\newcommand{\nn}{\nonumber}
\newcommand{\R}{\mathbb{R}}
\newcommand{\N}{\mathbb{N}}
\newcommand{\pO}{\partial\Omega}

\newcommand{\eps}{\varepsilon}

\newcommand{\supp}{{\rm supp} \, }

\newcommand{\wto}{\rightharpoonup}
\newcommand{\wsto}{\stackrel{\star}{\rightharpoonup}}
\newcommand{\hra}{\hookrightarrow}
\newcommand{\io}{\int_\Omega}
\newcommand{\na}{\nabla}
\newcommand{\Del}{\Delta}
\newcommand{\lam}{\lambda}
\newcommand{\pa}{\partial}
\newcommand{\bom}{\overline{\Omega}}
\newcommand{\Om}{\Omega}

\newcommand{\ov}{\overline}
\newcommand{\wh}{\widehat}

\newcommand{\hs}{\hspace*}
\newcommand{\cd}{\cdot}

\newcommand{\vp}{\varphi}

\newcommand{\abs}{\\[5pt]}

\newcommand{\tm}{T_{max}}
\newcommand{\tme}{T_{max,\eps}}
\newcommand{\ueps}{u_\eps}
\newcommand{\veps}{v_\eps}
\newcommand{\Feps}{F_\eps}
\newcommand{\beps}{b_\eps}
\newcommand{\heps}{h_\eps}
\newcommand{\yeps}{y_\eps}
\newcommand{\te}{\tau_\eps}
\newcommand{\vst}{v_\star}

\usepackage{newunicodechar}
\newunicodechar{ε}{\varepsilon}
\newunicodechar{α}{\alpha}
\newunicodechar{β}{\beta}
\newunicodechar{Ω}{\Omega}
\newunicodechar{Δ}{\Delta}
\newunicodechar{∇}{\nabla}
\newunicodechar{∂}{\partial}
\newunicodechar{γ}{\gamma}
\newunicodechar{Φ}{\Phi}
\newunicodechar{ℝ}{\R}
\newunicodechar{ℕ}{\mathbb{N}}
\newunicodechar{ν}{\nu}
\newunicodechar{∞}{\infty}
\newunicodechar{τ}{\tau}

\newcommand{\Ombar}{\overline{Ω}}
\newcommand{\norm}[2][]{\|#2\|_{#1}}
\newcommand{\matr}[1]{\begin{pmatrix}#1\end{pmatrix}}

\allowdisplaybreaks
\begin{document}
\enlargethispage{10mm}
\title{Radial solutions to a chemotaxis-consumption model\\ 
involving prescribed signal concentrations on the boundary}
\author{
Johannes Lankeit\footnote{lankeit@ifam.uni-hannover.de}\\
{\small Leibniz Universität Hannover, Institut für Angewandte Mathematik}\\
{\small Welfengarten~1, 30167 Hannover, Germany}
\and
Michael Winkler\footnote{michael.winkler@math.uni-paderborn.de}\\
{\small Institut f\"ur Mathematik, Universit\"at Paderborn,  }\\
{\small Warburger Str.~100, 33098 Paderborn, Germany}
}
\date{}
\maketitle
\begin{abstract}
\noindent
  The chemotaxis system 
  \begin{align*}
	u_t &= \Delta u - \nabla \cdot (u\nabla v), \\
	v_t &= \Delta v - uv,
  \end{align*}
  is considered under the boundary conditions $\frac{\partial u}{\partial\nu}- u\frac{\partial v}{\partial\nu}=0$ and $v=v_\star$ on $\partial\Omega$,
  where $\Omega\subset\mathbb{R}^n$ is a ball and $v_\star$ is a given positive constant.\abs
  In the setting of radially symmetric and suitably regular initial data, a result on global existence of bounded classical
  solutions is derived in the case $n=2$, while global weak solutions are constructed when $n\in \{3,4,5\}$.
  This is achieved by analyzing an energy-type inequality reminiscent of global structures previously observed 
  in related homogeneous Neumann problems. Ill-signed boundary integrals newly appearing therein are controlled
  by means of spatially localized smoothing arguments revealing higher order regularity features outside the spatial origin.\abs
  Additionally, unique classical solvability in the corresponding stationary problem is asserted, even in nonradial frameworks.\abs
\end{abstract}
%
%
%35Q92 (2010-now) PDEs in connection with biology, chemistry and other natural sciences 
%35Q35 (1991-now) PDEs in connection with fluid mechanics 
% 35K61 (2010-now) Nonlinear initial, boundary and initial-boundary value problems for nonlinear parabolic equations 
% 35K55 (1973-now) Nonlinear parabolic equations 
%35J61 (2010-now) Semilinear elliptic equations 
%
%
%
%
\newpage
\mysection{Introduction}\label{intro}
Chemotaxis systems, if posed in bounded domains, are usually studied with homogeneous Neumann boundary conditions. Especially where the chemotactic agents partially direct their motion toward higher concentrations of a signal which they consume instead of produce, however, other boundary conditions may become relevant.\abs
In this article, we consider the chemotaxis consumption system 
\be{firstsystem}
	\left\{ \begin{array}{ll}
	u_t &= \Delta u - \na \cdot (u\na v), \\
	v_t &= \Delta v - uv,
	\end{array} \right.
\ee
posing Dirichlet boundary conditions for the signal concentration $v$ and no-flux conditions for the bacterial population density $u$. \abs
Arising from a line of investigations concerned with pattern formation in colonies of {\em B. subtilis} in a fluid environment \cite{tuval}, such chemotaxis systems with signal consumption, additionally coupled to a Stokes- or Navier-Stokes fluid have been studied extensively over the last decade (for pointers to the literature see e.g. Section 4.1 of the survey \cite{BBTW} or the introduction of \cite{cao_lankeit}; for the model in the context of coral spawing, see e.g. \cite{zheng_corals}). While most works prescribed no-flux, homogeneous Neumann and homogeneous Dirichlet conditions for bacterial population density, signal concentration and fluid velocity, 
it turned out that these 
failed to adequately capture the colourful dynamics 
observed in the form of 
the patterns previously alluded to.\abs
In particular, a common result for the long-term behaviour was convergence to a constant state (i.e. a state without any patterns), as e.g. in \cite{taowin_evsmooth_stabil_consumption}, \cite{zhang_li}, \cite{win_arma}, \cite{win_TRAN}, in 
\cite{lankeit_m3as}, \cite{lankeit_wang} and \cite{win_food_supported}
in a system additionally including population growth, or also \cite{fan_jin} in a related system with nonlinear diffusion.\abs
Not least because of this, 
it has been suggested to use different, more realistic, inhomogeneous boundary conditions for the chemical signal
(see \cite{braukhoff} and \cite{braukhoff_lankeit}, but also \cite{tuval}), namely either  
Robin type boundary conditions where the rate of oxygen influx is controlled by the local oxygen concentration at the boundary or nonzero Dirichlet conditions directly prescribing the latter. 
(It has been confirmed \cite[Prop.~5.1]{braukhoff_lankeit} that the latter kind of conditions arises as a limit case of the former.)\abs
In general, altering boundary conditions can have a profound impact on the solution behaviour in chemotaxis systems, see the appearance
of a second critical mass in the Keller--Segel type system with Diricihlet conditions for $v$ studied in \cite{flw} if compared to the 
same system with Neumann conditions; in both cases homogeneous.
While -- in these particular settings related to \eqref{firstsystem} -- more realistic from a modelling perspective, the change of boundary conditions to inhomogeneous conditions brings about additional mathematical challenges. \abs
In particular, a large part of the analysis of \eqref{firstsystem} and its relatives relies on certain energy-like structures such 
as that expressed in the inequality
\begin{equation}\label{energy}
 \frac{d}{dt} \left(\io u\ln u + 2 \io |\nabla \sqrt{v}|^2 \right) + \io \frac{|\nabla u|^2}u + \io v|D^2\ln v|^2 + \frac12 \io \frac{u|\nabla v|^2}v \le 0,
\end{equation}
as documented for the Neumann problem for (\ref{firstsystem}) in \cite[(3.1)]{taowin_evsmooth_stabil_consumption}.
Replacing homogeneous Neumann by inhomogeneous Dirichlet boundary conditions for $v$, in the derivation of \eqref{energy} additional boundary integrals arise, destroying the (quasi-)Lyapunov structure of \eqref{energy} or relatives thereof on which existence results in, e.g., \cite{lankeit_m3as}, \cite{taowin_locallybounded}, \cite{win_weak_ctNS}, \cite{black_stokes_limit}, \cite{zhang_li}, \cite{taowin_evsmooth_stabil_consumption}, \cite{win_CPDE12} or also \cite{jiang_han} essentially relied.  \abs
Accordingly, only few results for chemotaxis-consumption models with boundary conditions different from homogeneous Neumann conditions
are available: Those concerning the related system with slightly different chemotaxis and energy consumption studied in \cite{knosalla_global,knosalla_nadzieja_stationary} with inhomogeneous Neumann and Dirichlet conditions are restricted to spatially one-dimensional domains.
For Robin-type conditions of the form introduced in \cite{braukhoff}, also in higher dimensions, the stationary problem of \eqref{firstsystem} has been shown to be uniquely solvable (for any prescribed total mass $\io u$ of the first component, see \cite{braukhoff_lankeit}), and (under a moderate smallness condition) features as the limit of a parabolic-elliptic simplification of \eqref{firstsystem} (cf.~\cite{FLM}). Also in fluid-coupled systems solutions have been found in the presence of logistic source terms 
(\cite{braukhoff}), or superlinear diffusion (\cite{wu_xiang,tian_xiang}), or without both (\cite{braukhoff_tang}).\abs
In the Dirichlet setting this article is concerned with 
we mention \cite{lee_wang_yang} where an asymptotic analysis of the vanishing diffusivity limit for $v$ in the stationary system seems to confirm the potential of \eqref{firstsystem} %with inhomogeneous Dirichlet conditions for $v$
to capture pattern dynamics. 
In the time-dependent problem (including fluid flow), solutions in $ℝ^2\times[0,1]$ were constructed in \cite{peng_xiang} if the signal consumption was strong, at least quadratic with respect to $u$, and in $\Omega\subset ℝ^2$ in \cite{wang_win_xiang10}, in both cases under a smallness condition on the initial data.
% $u^{(0)}$ in $L^1(\Omega)$ and $v^{(0)}$ in $L^\infty(\Omega)$ for the latter work. 
Without smallness conditions, solutions to \eqref{firstsystem} coupled to a fluid flow governed by the Stokes equations were constructed in \cite{wang_win_xiang6} ($\Omega\subset ℝ^N$ with linear diffusion for $N=2$ and porous medium type diffusion in  higher dimensions) and in \cite{wang_win_xiang_CPDE} ($\Omega\subset ℝ^3$). Nevertheless, the solution concepts pursued in these works are rather weak 
and do not yield comparable regularity as \cite{win_CPDE12} and \cite{tao_bdoxygenconsumption} for solutions to the system with homogeneous Neumann conditions.\abs
As to the above-mentioned difficulties concerning \eqref{energy}, different strategies have been employed: 
Exploiting the Robin condition in their systems, the works \cite{braukhoff} and subsequently \cite{wu_xiang} and \cite{tian_xiang} rely on a Lions-Magenes type transformation converting $v$ to a function with homogeneus Neumann boundary conditions. %Moreover, in these systems, there are additional boundedness-enhancing components.
The energy functional is enhanced by additional ``boundary energy'' terms in \cite{braukhoff_tang}.
In \cite{peng_xiang} and \cite{FLM}, a trace theorem is used to control the boundary integrals by integrals over the domain involving higher derivatives, which are available either due to the simpler elliptic form of the second equation (in \cite{FLM}) or due to a smallness condition (\cite{peng_xiang}, also in the result on long-term limit in \cite{FLM}).
In \cite{wang_win_xiang_CPDE}, a localized modification of the energy functional was investigated, the localization being detrimental to the regularity information near the boundary. Leaving \eqref{energy} behind, the approaches in
\cite{wang_win_xiang6} and \cite{wang_win_xiang10} used  different energy functionals (or small-data energy functionals), giving rise to less potent a priori estimates, as reflected in the generalized sense of solvability obtained.\abs
{\bf Regularity control on the boundary for radial solutions. Main results.} \quad
In this article, we plan to use radial symmetry as a mean to unravel difficulties related to possible effects
that the change from Neumann to Dirichlet boundary conditions for the signal may have on boundary regularity of solutions.
Specifically, in a ball $\Omega=B_R(0)\subset\R^n$ with $R>0$ and $n\ge 2$, and with a given positive constant $\vst$,
we shall consider the initial-boundary value problem
\be{0}
	\left\{ \begin{array}{ll}
	u_t = \Delta u - \na \cdot (u\na v), 
	\qquad & x\in \Omega, \ t>0, \\[1mm]
	v_t = \Delta v - uv, 
	\qquad & x\in \Omega, \ t>0, \\[1mm]
	\frac{\partial u}{\partial\nu}-u \frac{\partial v}{\partial\nu}=0, \quad v=\vst,
	\qquad & x\in \pO, \ t>0, \\[1mm]
	u(x,0)=u^{(0)}(x), \quad v(x,0)=v^{(0)}(x), 
	\qquad & x\in\Omega,
	\end{array} \right.
\ee
assuming that
\be{init}
	\left\{ \begin{array}{l}
	u^{(0)}\in W^{1,\infty}(\Om) \mbox{ is nonnegative in $\Om$ and radially symmetric with $u^{(0)}\not\equiv 0$,\quad  and} \\[1mm]
	v^{(0)}\in W^{1,\infty}(\Om) \mbox{ is positive in $\bom$ and radially symmetric with $v^{(0)}=\vst$ on $\pO$,}
	\end{array} \right.
\ee
where, as throughout the sequel, radial symmetry of a function $\varphi$ on $\Om$ is to be understood as referring to the 
spatial origin.\abs
To make appropriate use of these symmetry assumptions,
at a first stage of our analysis 
we shall rely on the essentially one-dimensional framework thereby generated in order to step by step turn the
basic properties of mass conservation in the first component, and uniform $L^\infty$ boundedness in the second,
into knowledge on higher order regularity features locally outside the spatial origin (see Section \ref{sect3} and especially
Corollary \ref{cor10}).
This will particularly enable us to appropriately control boundary integrals which due to the presence of possibly nonzero
normal derivatives arise in a spatially global energy analysis related to that in (\ref{energy}) (Section \ref{sect4}).\abs
In the spatially planar case, this will be found to entail a priori bounds actually in $L^\infty \times W^{1,\infty}$,
and to thus imply the following statement on global classical solvability and boundedness in (\ref{0}):
\begin{theo}\label{theo16}
  Let $R>0$ and $\Omega=B_R(0)\subset\R^2$, and suppose that $\vst\ge 0$, and that $u^{(0)}$ and $v^{(0)}$ satisfy (\ref{init}).
  Then there exist unique functions  
  \bas
	\left\{ \begin{array}{l}
	u\in C^0(\bom\times [0,\infty)) \cap C^{2,1}(\bom\times (0,\infty)) \qquad \mbox{and} \\[1mm]
	v\in \bigcap_{q>2} C^0([0,\infty);W^{1,q}(\Om)) \cap C^{2,1}(\bom\times (0,\infty))
	\end{array} \right.
  \eas
  which are such that $u>0$ and $v>0$ in $\bom\times [0,\infty)$, that $(u(\cdot,t),v(\cdot,t))$ is radially symmetric for all $t>0$,
  and that $(u,v)$ solves (\ref{0}) in the classical sense in $\Om\times (0,\infty)$.
  Moreover, there exists $C>0$ such that
  \be{16.1}
	\|u(\cdot,t)\|_{L^\infty(\Om)} + \|v(\cdot,t)\|_{W^{1,\infty}(\Om)} \le C
	\qquad \mbox{for all } t>0.
  \ee
\end{theo}
But also in some higher-dimensional situations, the regularity information gained from our energy analysis can be used to
establish a result on global solvability, albeit in a slightly weaker framework:
\begin{theo}\label{theo25}
  Let $n\in \{3,4,5\}$, $R>0$ and $\Omega=B_R(0)\subset\R^n$, let $\vst\ge 0$, and assume (\ref{init}).
  Then one can find nonnegative functions
  \be{25.01}
	\left\{ \begin{array}{l}
	u\in L^\infty((0,\infty);L^1(\Om)) \cap L^\frac{n+2}{n}_{loc}(\bom\times [0,\infty)) 
		\cap L^\frac{n+2}{n+1}_{loc}([0,\infty);W^{1,\frac{n+2}{n+1}}(\Om)) \qquad \mbox{and} \\[1mm]
	v\in L^\infty(\Om\times (0,\infty)) \quad \mbox{with} \quad
		v-\vst \in L^\infty((0,\infty);W_0^{1,2}(\Om)) \cap L^4_{loc}([0,\infty);W^{1,4}(\Om))
	\end{array} \right.
  \ee
  such that $u(\cdot,t)$ and $v(\cdot,t)$ are radially symmetric for a.e.~$t>0$, and that $(u,v)$ forms a global weak solution of
  (\ref{0}) in the sense of Definition \ref{dw} below. 
  Furthermore, there exists $C>0$ such that
  \be{25.1}
	\io u(\cdot,t)\ln u(\cdot,t) \le C
	\qquad \mbox{for almost all } t>0
  \ee
  and
  \be{25.2}
	\io |\na v(\cdot,t)|^2 \le C
	\qquad \mbox{for almost all } t>0
  \ee
  as well as
  \be{25.3}
	\int_t^{t+1} \io \Big\{ u^\frac{n+2}{n} + |\na u|^\frac{n+2}{n+1} + |\na v|^4 \Big\} \le C
	\qquad \mbox{for all } t>0.
  \ee
\end{theo}
We have to leave open here the question how far information on the large time behaviour of the above solutions
that goes beyond the boundedness features in (\ref{16.1}) and in (\ref{25.1})-(\ref{25.3}) can be derived,
especially in the presence of large initial data.
After all, a steady state analysis guided by the approach developed in \cite{braukhoff_lankeit} provides the following
result which may be viewed as an indication for nontrivial dynamics involving structured states in (\ref{0}):
\begin{theo}\label{th:stationary}
  Let $Ω\subset ℝ^n$ be a bounded domain with smooth boundary, and suppose that 
  $\vst \in \bigcup_{\beta\in (0,1)} C^{2+\beta}(\bom)$ is nonnegative.
  Then for every $m\ge 0$, the stationary problem \eqref{stationary} has a unique solution $(u,v)\in (C^2(\Ombar))^2$ 
  satisfying $\io u=m$.
  If $Ω=B_R(0)$ and $\vst$ is constant, then this solution is radially symmetric, and both $u$ and $v$ are convex.
\end{theo}
\mysection{Local solvability, approximation and basic properties}
In order to simultaneously address, throughout large parts of our analysis, 
the case $n=2$ in which classical solvability is strived for, and the case $n\in \{3,4,5\}$
in which we intend to construct a solution via approximation, for $\eps\in [0,1)$ let us consider the variants of (\ref{0}) 
given by
\be{0eps}
	\left\{ \begin{array}{ll}
	u_{\eps t} = \Delta\ueps - \na \cdot \big( \ueps\Feps'(\ueps)\na\veps\big),
	\qquad & x\in\Om, \ t>0, \\[1mm]
	v_{\eps t} = \Delta \veps - \Feps(\ueps)\veps,
	\qquad & x\in\Om, \ t>0, \\[1mm]
	\frac{\partial\ueps}{\partial\nu} - \ueps\Feps'(\ueps)\frac{\partial\veps}{\partial\nu}=0, \quad \veps=\vst,
	\qquad & x\in\pO, \ t>0, \\[1mm]
	\ueps(x,0)=u^{(0)}(x), \quad \veps(x,0)=v^{(0)}(x),
	\qquad & x\in\Om,
	\end{array} \right.
\ee
where
\be{F0}
	\Feps(\xi):=\frac{\xi}{1+\eps\xi},
	\qquad \xi\ge 0, \ \eps\in [0,1),
\ee
satisfies 
\be{F}
	0 \le \Feps(\xi)\le \xi
	\quad \mbox{and} \quad
	0 \le \Feps'(\xi)=\frac{1}{(1+\eps\xi)^2} \le 1
	\qquad \mbox{for all $\xi\ge 0$ and } \eps\in [0,1);
\ee
indeed, these choices ensure that (\ref{0eps}) coincides with (\ref{0}) when $\eps=0$.\abs
Local solvability and a handy extensibility criterion can be obtained by resorting to standard literature:
\begin{lem}\label{lem1}
  Let $\eps\in [0,1)$. Then there exist $\tme\in (0,\infty]$ and uniquely determined functions
  \bas
	\left\{ \begin{array}{l}
	\ueps\in C^0(\bom\times [0,\tme)) \cap C^{2,1}(\bom\times (0,\tme)) \qquad \mbox{and} \\[1mm]
	\veps\in \bigcap_{q>n} C^0([0,\tme);W^{1,q}(\Om)) \cap C^{2,1}(\bom\times (0,\tme))
	\end{array} \right.
  \eas
  such that $\ueps>0$ and $\veps>0$ in $\bom\times [0,\tme)$, that $\ueps(\cdot,t)$ and $\veps(\cdot,t)$ are radially symmetric 
  for all $t\in (0,\tme)$,
  that $(\ueps,\veps)$ solves (\ref{0eps}) classically in $\Om\times (0,\tme)$, and that
  \be{ext}
	\mbox{either $\tme=\infty$, \quad or \quad}
	\limsup_{t\nearrow\tme} \Big\{ \|\ueps(\cdot,t)\|_{W^{1,q}(\Om)} + \|\veps(\cdot,t)\|_{W^{1,q}(\Om)} \Big\}
	= \infty
	\quad
	\mbox{for all $q>n$}.
  \ee
\end{lem}
\begin{proof}
  This results from \cite[Thm.~14.4 and Thm.~14.6]{amann_nonhomogeneous} when, for $U=\matr{u\\v-\vst}$, applied to 
  the evolution problem given by $U_t=\nabla\cdot (A(U)\nabla U)+f(U)$, 
  $\left(\matr{1&0\\0&0} A(U)ν + \matr{0&0\\0&1} U\right)\vert_{∂Ω} =0$, $U(0)=\matr{u^{(0)}\\v^{(0)}-\vst}$, 
  with $A(U)=\matr{1&-U_1F'(U_1)\\0&1}$, $f(U)=\matr{0\\-F(U_1)(U_2+\vst)}$.
\end{proof}
These solutions clearly preserve mass in their first component and are bounded in the second.
\begin{lem}\label{lem2}
  Let $\eps\in [0,1)$. Then 
  \be{mass}
	\io \ueps(\cdot,t)=\io u^{(0)}
	\qquad \mbox{for all } t\in (0,\tme)
  \ee
  and
  \be{vinfty}
	\|\veps(\cdot,t)\|_{L^\infty(\Om)} \le \|v^{(0)}\|_{L^\infty(\Om)}
	\qquad \mbox{for all } t\in (0,\tme).
  \ee
\end{lem}
\begin{proof}
  While (\ref{mass}) can directly be seen upon integrating the first equation in (\ref{0eps}), the inequality in (\ref{vinfty})
  can be verified by means of the comparison principle applied to the second equation in (\ref{0eps}), because
  $\ov{v}(x,t):=\|v^{(0)}\|_{L^\infty(\Om)}$, $(x,t)\in\bom\times [0,\infty)$, satisfies 
  $\ov{v}_t- \Delta\ov{v}+\Feps(\ueps)\ov{v}=\Feps(\ueps)\ov{v}\ge 0$ in $\Om\times (0,\tme)$ for all $\eps\in (0,1)$ as well as
  $\ov{v}|_{t=0}\ge v^{(0)}$ and $\ov{v}|_{\pO}\ge \vst$, the latter due to the fact that (\ref{init}) necessarily requires that
  $\vst \le \|v^{(0)}\|_{L^\infty(\Om)}$.
\end{proof}
Also for the gradient of the second solution component some first a priori estimates are available. 
\begin{lem}\label{lem4}
  There exists $C>0$ such that
  \be{4.1}
	\|\na\veps(\cdot,t)\|_{L^1(\Om)} \le C
	\qquad \mbox{for all $t\in (0,\tme)$ and $\eps\in [0,1)$.}
  \ee
\end{lem}
\begin{proof}
  According to well-known smoothing estimates for the Dirichlet heat semigroup $(e^{t\Delta})_{t\ge 0}$ on $\Om$
  (\cite{eidelman_ivasishen}, \cite[Section 48.2]{quittner_souplet}), there exist $\lam>0$, $c_1>0$ and $c_2>0$ such that for all $t>0$,
  \bas
	\|\na e^{t\Delta}\varphi\|_{L^1(\Om)} \le c_1 \|\varphi\|_{W^{1,\infty}(\Om)}
	\qquad \mbox{for all $\varphi \in W^{1,\infty}(\Om)$ such that $\varphi=\vst$ on $\pO$,}
  \eas
  and 
  \bas
	\|\na e^{t\Delta}\varphi\|_{L^1(\Om)} \le c_2 \cdot (1+t^{-\frac{1}{2}}) e^{-\lam t} \|\varphi\|_{L^1(\Om)}
	\qquad \mbox{for all $\varphi \in C^0(\bom)$ such that $\varphi=\vst$ on $\pO$.}
  \eas
  Since (\ref{F}) together with (\ref{mass}) and (\ref{vinfty}) ensures that for all $\eps\in [0,1)$ we have
  $0\le \Feps(\ueps)\le \ueps$ and hence
  \bas
	\|\Feps(\ueps)\veps\|_{L^1(\Om)}
	\le \|\Feps(\ueps)\|_{L^1(\Om)} \|\veps\|_{L^\infty(\Om)}
	\le \|\ueps\|_{L^1(\Om)} \|\veps\|_{L^\infty(\Om)}
	\le c_3:=\|u^{(0)}\|_{L^1(\Om)} \|v^{(0)}\|_{L^\infty(\Om)}
%	\qquad \mbox{for all } t\in (0,\tme),
  \eas
  for all $t\in (0,\tme)$, in view of (\ref{0eps}) this implies that for any such $\eps$,
  \bas
	\|\na\veps(\cdot,t)\|_{L^1(\Om)}
	&=& \big\| \na \big(\veps(\cdot,t)-\vst\big)\big\|_{L^1(\Om)} \\
	&=& \bigg\| \na e^{t\Delta} (v^{(0)}-\vst)
	- \int_0^t \na e^{(t-s)\Delta} \Big\{ \Feps\big(\ueps(\cdot,s)\big)\veps(\cdot,s) \Big\} ds \bigg\|_{L^1(\Om)} \\
	&\le& c_1 \|v^{(0)}-\vst\|_{W^{1,\infty}(\Om)}
	+ c_2 c_3 \int_0^t \Big(1+(t-s)^{-\frac{1}{2}}\Big) e^{-\lam (t-s)} ds \\
	&\le& c_1 \|v^{(0)}-\vst\|_{W^{1,\infty}(\Om)}
	+ c_2 c_3 \int_0^\infty (1+\sigma^{-\frac{1}{2}}) e^{-\lam\sigma} d\sigma
	\qquad \mbox{for all } t\in (0,\tme),
  \eas
  which establishes (\ref{4.1}).
\end{proof}
\mysection{Local estimates outside the origin}\label{sect3}
%
%
%
%
%
%
%
%
%According to Lemma~\ref{lem1}, $(\ueps(\cdot,t),\veps(\cdot,t))$ is radially symmetric
%for every $ε\in[0,1)$ and $t\in (0,\tme)$. 
In line with common abuse of notation, we occasionally write $\ueps(r,t)$ and $\veps(r,t)$,
instead of $\ueps(x,t)$ and $\veps(x,t)$, for $r=|x|\in[0,R]$, $t\in[0,\tme)$ and $\eps\in [0,1)$,
and in order to prepare our derivation of local estimates outside the origin, we 
observe that whenever $\chi\in C_0^\infty((0,R])$, for each $\eps\in [0,1)$ we have
\be{chiv}
	\Big( \chi \cdot (\veps-\vst)\Big)_t
	= \Big( \chi\cdot (\veps-\vst)\Big)_{rr}
	+ \beps(r,t),
	\qquad r\in (0,R), \ t\in (0,\tme),
\ee
where
\bea{b}
	& & \hs{-20mm}
	\beps(r,t) \equiv \beps^{(\chi)}(r,t)
	:= \Big( \frac{n-1}{r}\chi(r)-2\chi_r(r)\Big) v_{\eps r}(r,t) 
	- \chi_{rr}(r) \cdot \big(\veps(r,t)-\vst\big) \nn\\
	& & \hs{40mm}
	- \chi(r)\cdot \Feps\big(\ueps(r,t)\big) \veps(r,t),
	\qquad r\in (0,R), \ t\in (0,\tme).
\eea
As the above basic estimates imply $L^1$ bounds for $\beps$, a straightforward argument based on parabolic smoothing
in the one-dimensional heat equation (\ref{chiv}) yields the following information on regularity of the taxis gradient 
outside the origin.
\begin{lem}\label{lem5}
  Let $q\in (1,\infty)$ and $\delta\in (0,R)$. Then there exists $C(q,\delta)>0$ with the property that
  \be{5.1}
	\|v_{\eps r}(\cdot,t)\|_{L^q((\delta,R))} \le C(q,\delta)
	\qquad \mbox{for all $t\in (0,\tme)$ and $\eps\in [0,1)$.}
  \ee
\end{lem}
\begin{proof}
  Given $\delta\in (0,R)$, we fix $\chi\in C^\infty([0,R])$ such that $0\le\chi\le 1$, that $\chi\equiv 0$ in $[0,\frac{\delta}{2}]$,
  and that $\chi\equiv 1$ in $[\delta,R]$, and let $(e^{-tA})_{t\ge 0}$ denote the one-dimensional heat semigroup generated
  by the operator $A:=-(\cdot)_{rr}$ under homogeneous Dirichlet boundary conditions on $(\frac{\delta}{2},R)$.
  Then known regularization features of the latter (\cite{eidelman_ivasishen}, \cite{quittner_souplet}) ensure that if we fix 
  $q\in (1,\infty)$, then we can find $\lam=\lam(q,\delta)>0, c_1=c_1(q,\delta)>0$ and $c_2=c_2(q,\delta)>0$ such that whenever $t>0$,
  \be{5.2}
	\|\partial_r e^{-tA}\varphi\|_{L^q((\frac{\delta}{2},R))}
	\le c_1\|\varphi\|_{W^{1,\infty}((\frac{\delta}{2},R))}
	\qquad \mbox{for all $\varphi\in W^{1,\infty}((\frac{\delta}{2},R))$ such that $\varphi(\frac{\delta}{2})=\varphi(R)=0$}
  \ee
  and
  \be{5.3}
	\|\partial_r e^{-tA}\varphi\|_{L^q((\frac{\delta}{2},R))}
	\le c_2 \cdot \big(1+t^{-1+\frac{1}{q}}\big) e^{-\lam t} \|\varphi\|_{L^1((\frac{\delta}{2},R))}
	\qquad \mbox{for all $\varphi\in C^0([\frac{\delta}{2},R])$ with $\varphi(\frac{\delta}{2})=\varphi(R)=0$}.
  \ee
  Apart from that, a combination of Lemma \ref{lem4} with (\ref{mass}) and (\ref{vinfty}) shows that since 
  $\supp \chi \subset [\frac{\delta}{2},R]$, we can pick $c_3=c_3(\delta)>0$ in such a way that for any $\eps\in [0,1)$, the 
  function $\beps=\beps^{(\chi)}$ defined in (\ref{b}) satisfies
  \bas
	\|\beps(\cdot,t)\|_{L^1((\frac{\delta}{2},R))} \le c_3
	\qquad \mbox{for all } t\in (0,\tme).
  \eas
  On the basis of (\ref{chiv}) and the fact that $\chi \cdot (\veps-\vst)=0$ on $\{\frac{\delta}{2},R\} \times (0,\tme)$,
  we can therefore utilize (\ref{5.2}) and (\ref{5.3}) to estimate
  \bea{5.4}
	\Big\| \partial_r \Big\{ \chi \cdot \big(\veps(\cdot,t)-\vst\big) \Big\} \Big\|_{L^q((\frac{\delta}{2},R))}
	&=& \bigg\| \partial_r e^{-tA} \Big\{ \chi \cdot (v^{(0)}-\vst)\Big\}
	+ \int_0^t \partial_r e^{-tA} \beps(\cdot,s) ds \bigg\|_{L^q((\frac{\delta}{2},R))} \nn\\
	&\le& c_1 \big\| \chi \cdot (v^{(0)}-\vst)\big\|_{W^{1,\infty}((\frac{\delta}{2},R))} \nn\\
	& & + c_2 \int_0^t \Big(1+(t-s)^{-1+\frac{1}{q}}\Big) e^{-\lam (t-s)} \|\beps(\cdot,s)\|_{L^1((\frac{\delta}{2},R))} ds \nn\\
	&\le& c_1 \big\| \chi \cdot (v^{(0)}-\vst)\big\|_{W^{1,\infty}((\frac{\delta}{2},R))}
	+ c_2 c_3 c_4
%	\qquad \mbox{for all $t\in (0,\tme)$ and } \eps\in [0,1),
  \eea
  for all $t\in (0,\tme)$ and $\eps\in [0,1)$,
  with $c_4:=\int_0^\infty(1+\sigma^{-1+\frac{1}{q}}) e^{-\lam\sigma} d\sigma$ being finite since we are assuming that $q<\infty$.
  As $\chi\equiv 1$ in $[\delta,R]$, (\ref{5.1}) directly results from (\ref{5.4}).
\end{proof}
This has consequences for bounds on certain powers of $\ueps$ and their derivative outside a neighbourhood of the spatial origin.
\begin{lem}\label{lem6}
  For any choice of $p\in (0,1)$ and $\delta\in (0,R)$ one can find $C(p,\delta)>0$ such that whenever $\eps\in [0,1)$,
  \be{6.1}
	\int_t^{t+\te} \int_\delta^R \Big| \big(\ueps^\frac{p}{2}\big)_r \Big|^2 \le C(p,\delta)
	\qquad \mbox{for all $t\in [0,\tme-\te)$,}
  \ee
  and that
  \be{6.2}
	\int_t^{t+\te} \int_\delta^R \ueps^{p+2} \le C(p,\delta)
	\qquad \mbox{for all $t\in [0,\tme-\te)$,}
  \ee
  where $\te:=\min\{1,\frac{1}{2}\tme\}$.
\end{lem}
\begin{proof}
  We fix $\zeta\in C^\infty(\bom)$ such that $0\le\zeta\le 1$, and that $\zeta\equiv 0$ in $\ov{B}_\frac{\delta}{2}(0)$
  as well as $\zeta\equiv 1$ in $\bom\setminus \ov{B}_\delta(0)$.
  Relying on the positivity of $\ueps$ in $\bom\times (0,\tme)$ for $\eps\in [0,1)$, from (\ref{0eps}) we then obtain that due to
  Young's inequality,
  \bea{6.3}
	\frac{1}{p} \frac{d}{dt} \io \zeta^2 \ueps^p
	&=& \io \zeta^2 \ueps^{p-1} \na \cdot \Big\{ \na\ueps-\ueps\Feps'(\ueps)\na\veps\Big\} \nn\\
	&=& (1-p) \io \zeta^2 \ueps^{p-2} |\na\ueps|^2
	-(1-p) \io \zeta^2 \ueps^{p-1} \Feps'(\ueps) \na\ueps\cdot\na\veps \nn\\
	& & - 2\io \zeta \ueps^{p-1} \na\ueps\cdot\na\zeta 
	+  2 \io \zeta\ueps^p \Feps'(\ueps) \na\veps\cdot\na\zeta \nn\\
	&\ge& \frac{1-p}{2} \io \zeta^2 \ueps^{p-2} |\na\ueps|^2
	- (2-p) \io \zeta^2 \ueps^p |\na\veps|^2 \nn\\
	& & - \frac{5-p}{1-p} \io |\na\zeta|^2 \ueps^p
	\qquad \mbox{for all $t\in (0,\tme)$ and } \eps\in [0,1),
  \eea
  because $0\le\Feps'\le 1$ for all $\eps\in [0,1)$.
  Here by the H\"older inequality,
  \bas
	\io \zeta^2 \ueps^p |\na\veps|^2
	\le \bigg\{ \io \ueps \bigg\}^p \cdot \bigg\{ \io \zeta^\frac{2}{1-p} |\na\veps|^\frac{2}{1-p}\bigg\}^{1-p}
	\qquad \mbox{for all $t\in (0,\tme)$ and } \eps\in [0,1)
  \eas
  and
  \bas
	\io |\na\zeta|^2 \ueps^p
	\le \|\na\zeta\|_{L^\frac{2}{1-p}(\Om)}^2 \cdot \bigg\{ \io \ueps \bigg\}^p
	\qquad \mbox{for all $t\in (0,\tme)$ and } \eps\in [0,1),
  \eas
  so that since $\supp \zeta\subset \bom\setminus B_\frac{\delta}{2}(0)$ we may apply Lemma \ref{lem5} to $q:=\frac{2}{1-p}$ to see
  that thanks to (\ref{mass}), with some $c_1=c_1(p,\delta)>0$ we have
  \bas
	(2-p)\io \zeta^2 \ueps^p |\na\veps|^2
	+ \frac{5-p}{1-p} \io |\na\zeta|^2 \ueps^p
	\le c_1
	\qquad \mbox{for all $t\in (0,\tme)$ and } \eps\in [0,1).
  \eas
  Therefore, an integration in (\ref{6.3}) shows that again due to the H\"older inequality and (\ref{mass}),
  \bas
	\frac{1-p}{2} \int_t^{t+\te} \int_{\Om\setminus B_\delta(0)} \ueps^{p-2} |\na\ueps|^2
	&\le& \frac{1-p}{2} \int_t^{t+\te} \io \zeta^2 \ueps^{p-2} |\na\ueps|^2 \\
	&\le& \frac{1}{p} \io \zeta^2 \ueps^p(\cdot,t-\te)
	- \frac{1}{p} \io \zeta^2 \ueps^p(\cdot,t) 
	+ c_1\te \\
	&\le& \frac{1}{p} \cdot \bigg\{ \io u^{(0)} \bigg\}^p 
	+ c_1
	\quad \mbox{for all $t\in [0,\tme-\te)$ and } \eps\in [0,1),
  \eas
  because $\te\le 1$.
  This implies (\ref{6.1}), whereupon (\ref{6.2}) readily results from (\ref{6.1}) according to the fact that the Gagliardo-Nirenberg
  inequality provides $c_2=c_2(p,\delta)>0$ fulfilling
  \bas
	\int_\delta^R \ueps^{p+2}
	&=& \|\ueps^\frac{p}{2}\|_{L^\frac{2(p+2)}{p}((\delta,R))}^\frac{2(p+2)}{p} \\
	&\le& c_2\Big\| \big(\ueps^\frac{p}{2}\big)_r \Big\|_{L^2((\delta,R))}^2
		\big\|\ueps^\frac{p}{2}\big\|_{L^\frac{2}{p}((\delta,R))}^\frac{4}{p}
	+ c_2 \big\|\ueps^\frac{p}{2}\big\|_{L^\frac{2}{p}((\delta,R))}^\frac{2(p+2)}{p}
	\quad \mbox{for all $t\in (0,\tme)$ and } \eps\in [0,1),
  \eas
  and because for all $t\in (0,\tme)$ and $\eps\in [0,1)$,
  \bas
	\big\|\ueps^\frac{p}{2}\big\|_{L^\frac{2}{p}((\delta,R))}^\frac{2}{p}
	= \int_\delta^R \ueps(r,t) dr 
	\le \delta^{1-n} \int_0^R r^{n-1} \ueps(r,t) dr
	= \delta^{1-n} \int_0^R r^{n-1} u^{(0)}(r) dr
%	\qquad \mbox{for all $t\in (0,\tme)$ and } \eps\in [0,1)
  \eas
  by (\ref{mass}).
\end{proof}
In contrast to settings with homogeneous boundary conditions, in the present situation it will become necessary to deal with non-vanishing boundary terms. While this section will culminate in corresponding estimates, a key to these becomes visible in the following corollary already.
\begin{cor}\label{cor7}
  There exists $C>0$ such that
  \be{7.1}
	\int_t^{t+\te} \ueps(R,s) ds \le C
	\qquad \mbox{for all $t\in [0,\tme-\te)$,}
  \ee
  where again $\te=\min\{1,\frac{1}{2}\tme\}$ for $\eps\in [0,1)$.
\end{cor}
\begin{proof}
  By means of the Gagliardo-Nirenberg inequality, we can pick $c_1>0$ with the property that
  \bas
	\|\varphi\|_{L^\infty((\frac{R}{2},R))}^6
	\le c_1\|\varphi_r\|_{L^2((\frac{R}{2},R))}^2 \|\varphi\|_{L^4((\frac{R}{2},R))}^4
	+ c_1\|\varphi\|_{L^4((\frac{R}{2},R))}^6
	\qquad \mbox{for all $\varphi\in W^{1,2}((\frac{R}{2},R))$,}
  \eas
  whence
  \bas
	\int_t^{t+\te} \ueps^\frac{3}{2}(R,s) ds
	&\le& \int_t^{t+\te} \big\| \ueps^\frac{1}{4}(\cdot,s)\big\|_{L^\infty((\frac{R}{2},R))}^6 ds \\
	&\le& c_1 \int_t^{t+\te} \Big\| \big(\ueps^\frac{1}{4}\big)_r(\cdot,s) \Big\|_{L^2((\frac{R}{2},R))}^2
		\big\|\ueps^\frac{1}{4}(\cdot,s)\big\|_{L^4((\frac{R}{2},R))}^4 ds \\
	& & + c_1\int_t^{t+\te} \big\|\ueps^\frac{1}{4}(\cdot,s)\big\|_{L^4((\frac{R}{2},R))}^6 ds
	\qquad \mbox{for all $t\in [0,\tme-\te)$ and } \eps\in [0,1).
  \eas
  Combining (\ref{mass}) with an application of Lemma \ref{lem6} to $p:=\frac{1}{2}$ thus shows that with some $c_2>0$ we have
  \bas
	\int_t^{t+\te} \ueps^\frac{3}{2}(R,s) ds \le c_2
	\qquad \mbox{for all $t\in [0,\tme-\te)$ and } \eps\in [0,1),
  \eas
  from which (\ref{7.1}) follows upon employing the H\"older inequality.
\end{proof}
The following elementary observation, a proof of which can be found in \cite[Lemma 3.4]{win_JFA}, will be referred to
in Lemma~\ref{lem8}, Lemma \ref{lem11} and Lemma \ref{lem14}.
\begin{lem}\label{lem77}
  Let $T\in (0,\infty]$ and $\tau\in (0,T)$, and let $h\in L^1_{loc}((0,T))$ be nonnegative and such that
  \bas
	\int_t^{t+\tau} h(s) ds \le b
	\qquad \mbox{for all } t\in (0,T-\tau)
  \eas
  with some $b>0$. Then
  \bas
	\int_0^t e^{-\lam(t-s)} h(s) ds \le \frac{b\tau}{1-e^{-\lam\tau}}
	\qquad \mbox{for all $t\in (0,T)$ and any } \lam>0.
  \eas
\end{lem}
Whereas the previous estimates for $\ueps$ were concerned with temporally integrated quantities, the following lemma provides a temporally uniform bound.
\begin{lem}\label{lem8}
  Let $p\in (1,3)$ and $\delta\in (0,R)$. Then there exists $C(p,\delta)>0$ such that
  \be{8.1}
	\int_\delta^R \ueps^p(r,t) dr
	\le C(p,\delta)
	\qquad \mbox{for all $t\in (0,\tme)$ and } \eps\in [0,1).
  \ee
\end{lem}
\begin{proof}
  We again take a function $\zeta\in C^\infty(\bom)$ fulfilling $0\le\zeta\le 1$ and $\zeta|_{\ov{B}_\frac{\delta}{2}(0)}\equiv 0$
  as well as $\zeta|_{\bom\setminus B_\delta(0)}\equiv 1$, and once more rely on (\ref{0eps}) to see by means of Young's inequality
  and (\ref{F}) that
  \bea{8.2}
	\frac{d}{dt} \io \zeta^2 \ueps^p
	+ \io \zeta^2 \ueps^p
	&=& - p(p-1) \io \zeta^2 \ueps^{p-2} |\na\ueps|^2
	+ p(p-1) \io \zeta^2 \ueps^{p-1} \Feps'(\ueps) \na\ueps\cdot\na\veps \nn\\
	& & -2p \io \zeta \ueps^{p-2} \na\ueps\cdot\na\zeta
	+ 2p \io \zeta\ueps^p \Feps'(\ueps) \na\veps\cdot\na\zeta \nn\\
	& & + \io \zeta^2 \ueps^p \nn\\
	&\le& \frac{p(p+1)}{2} \io \zeta^2 \ueps^p |\na\veps|^2
	+ \frac{p(p+1)}{p-1} \io |\na\zeta|^2 \ueps^p \nn\\
	& & + \io \zeta^2 \ueps^p
	\qquad \mbox{for all $t\in (0,\tme)$ and } \eps\in [0,1).
  \eea
  Here, taking any $p_0=p_0(p)>p$ such that $p_0<3$, we may again draw on Young's inequality to estimate
  \bas
	\io \zeta^2 \ueps^p |\na\veps|^2
	\le \int_{\Omega\setminus B_\frac{\delta}{2}(0)} \ueps^{p_0}
	+ \int_{\Omega\setminus B_\frac{\delta}{2}(0)} |\na\veps|^\frac{2p_0}{p_0-p}
	\qquad \mbox{for all $t\in (0,\tme)$ and } \eps\in [0,1)
  \eas
  and
  \bas
	\io |\na\zeta|^2 \ueps^p
	\le \int_{\Omega\setminus B_\frac{\delta}{2}(0)} \ueps^{p_0}
	+ \io |\na\zeta|^\frac{2p_0}{p_0-p}
	\qquad \mbox{for all $t\in (0,\tme)$ and } \eps\in [0,1)
  \eas
  as well as 
  \bas
	\io \zeta^2 \ueps^p
	\le \int_{\Omega\setminus B_\frac{\delta}{2}(0)} \ueps^{p_0}
	+ |\Omega|
	\qquad \mbox{for all $t\in (0,\tme)$ and } \eps\in [0,1),
  \eas
  so that invoking Lemma \ref{lem5} we find $c_1=c_1(p,\delta)>0$ fulfilling
  \bas
	& & \hs{-30mm}
	\frac{p(p+1)}{2} \io \zeta^2 \ueps^p |\na\veps|^2
	+ \frac{p(p+1)}{p-1} \io |\na\zeta|^2 \ueps^p
	+ \io \zeta^2 \ueps^p \\
	&\le& \heps(t):=c_1\int_{\Omega\setminus B_\frac{\delta}{2}(0)} \ueps^{p_0}
	+ c_1
	\qquad \mbox{for all $t\in (0,\tme)$ and } \eps\in [0,1).
  \eas
  From (\ref{8.2}) we therefore obtain that
  \bas
	\frac{d}{dt} \io \zeta^2 \ueps^p + \io \zeta^2 \ueps^p
	\le \heps(t)
	\qquad \mbox{for all $t\in (0,\tme)$ and } \eps\in [0,1),
  \eas
  so that since $c_2=c_2(p,\delta):=\sup_{\eps\in [0,1)} \sup_{t\in (0,\tme-\te)} \int_t^{t+\te} \heps(s) ds$ with $\te=\min\{1,\frac{1}{2}\tme\}$ is finite according to Lemma \ref{lem6} and the fact that $p_0<3$, by using an ODE comparison argument along with  Lemma~\ref{lem77}
  we infer that
  \bas
	\io \zeta^2 \ueps^p
	&\le& e^{-t} \cdot \io \zeta^2 (u^{(0)})^p
	+ \int_0^t e^{-(t-s)} \heps(s) ds \\
	&\le& \io (u^{(0)})^p
	+ \frac{c_2\te}{1-e^{-\te}}
	\qquad \mbox{for all $t\in (0,\tme)$ and } \eps\in [0,1),
  \eas
  and hence conclude as intended, because $\frac{\tau}{1-e^{-\tau}} \le \frac{1}{1-e^{-1}}$ for all $\tau\in (0,1]$,  
  and because $\zeta\equiv 1$ in $\Om\setminus B_\delta(0)$.
\end{proof}
With these bounds at hand, we can even estimate the derivative of the second component uniformly, again
outside a neighbourhood of the origin.
\begin{lem}\label{lem9}
  For each $\delta\in (0,R)$ there exists $C(\delta)>0$ satisfying
  \be{9.1}
	|v_{\eps r}(r,t)| \le C(\delta)
	\qquad \mbox{for all $r\in [\delta,R]$, $t\in (0,\tme)$ and } \eps\in [0,1).
  \ee
\end{lem}
\begin{proof}
  We once more take $\chi\in C^\infty([0,R])$ such that $0\le\chi\le 1$ and $\chi|_{[0,\frac{\delta}{2}]}\equiv 0$
  as well as $\chi|_{[\delta,R]} \equiv 1$, and then infer from Lemma \ref{lem8} in conjunction with (\ref{F}), (\ref{vinfty}) and
  Lemma \ref{lem5} that there exists $c_1=c_1(\delta)>0$ such that with 
  $(\beps)_{\eps\in [0,1)} \equiv (\beps^{(\chi)})_{\eps \in [0,1)}$
  as defined in (\ref{b}) we have
  \bas
	\|\beps(\cdot,t)\|_{L^2((\frac{\delta}{2},R))} \le c_1
	\qquad \mbox{for all $t\in (0,\tme)$ and } \eps\in [0,1).
  \eas
  As for the Dirichlet heat semigroup $(e^{-tA})_{t\ge 0}$ on $(\frac{\delta}{2},R))$ it is 
  known (\cite{eidelman_ivasishen}, \cite{quittner_souplet}) that there exist $\lam=\lam(\delta)>0$,
  $c_2=c_2(\delta)>0$ and $c_3=c_3(\delta)>0$ with the property that for all $t>0$,
  \bas
	\|\partial_r e^{-tA} \varphi\|_{L^\infty((\frac{\delta}{2},R))}
	\le c_2\|\varphi\|_{W^{1,\infty}((\frac{\delta}{2},R))}
	\qquad \mbox{for all $\varphi\in W^{1,\infty}((\frac{\delta}{2},R))$ such that 
		$\varphi(\frac{\delta}{2})=\varphi(R)=0$}
  \eas
  as well as
  \bas
	\|\partial_r e^{-tA} \varphi\|_{L^\infty((\frac{\delta}{2},R))}
	\le c_3\cdot (1+t^{-\frac{3}{4}}) e^{-\lam t} \|\varphi\|_{L^2((\frac{\delta}{2},R))}
	\quad \mbox{for all $\varphi\in C^0([\frac{\delta}{2},R])$ with
		$\varphi(\frac{\delta}{2})=\varphi(R)=0$,}
  \eas
  from (\ref{chiv}) we obtain that
  \bas
	\|v_{\eps r}(\cdot,t)\|_{L^\infty((\delta,R))}
	&\le& \Big\|\partial_r \Big\{ \chi \cdot \big(\veps(\cdot,t)-\vst\big) \Big\} \Big\|_{L^\infty((\frac{\delta}{2},R))} \\
	&\le& c_2 \big\| \chi\cdot (v^{(0)}-\vst)\big\|_{W^{1,\infty}((\frac{\delta}{2},R))}
	+ c_1 c_3 \int_0^\infty (1+\sigma^{-\frac{3}{4}}) e^{-\lam\sigma} d\sigma
  \eas
  for all $t\in (0,\tme)$ and any $\eps\in [0,1)$.
\end{proof}
In conclusion, this allows for appropriately controlling all the boundary integrals that will turn out to appear
in the course of our subsequent energy analysis:
\begin{cor}\label{cor10}
  There exists $C>0$ such that
  \be{10.1}
	\int_t^{t+\te} \int_{\pO} \frac{1}{\veps} \frac{\partial |\na\veps|^2}{\partial\nu} \le C
	\qquad \mbox{for all $t\in [0,\tme-\te)$}
  \ee
  and 
  \be{10.2}
	\int_t^{t+\te} \int_{\pO} |\na\veps|^2 \frac{\partial |\na\veps|^2}{\partial\nu} \le C
	\qquad \mbox{for all $t\in [0,\tme-\te)$}
  \ee
  as well as
  \be{10.3}
	\bigg| \int_{\pO} \frac{|\na\veps|^2}{\veps^2} \frac{\partial\veps}{\partial\nu} \bigg| \le C
	\qquad \mbox{for all $t\in (0,\tme)$ and } \eps\in [0,1),
  \ee
  where, as before, $\te=\min\{1,\frac{1}{2}\tme\}$ for $\eps\in [0,1)$.
\end{cor}
\begin{proof}
  Once more explicitly relying on radial symmetry, we may use that according to Lemma \ref{lem1} the second equation in (\ref{0eps})
  holds up to $\pO$ throughout $(0,\tme)$, which namely ensures that on $\pO$ we have the one-sided inequality 
  \bas
	\frac{1}{\veps} \cdot \frac{\partial |\na\veps|^2}{\partial\nu}
	&=& \frac{2}{\veps} v_{\eps r} v_{\eps rr} \\
	&=& \frac{2}{\veps} v_{\eps r} \cdot \Big\{ v_{\eps rr} + \frac{n-1}{R} v_{\eps r} \Big\} 
	- \frac{2(n-1)}{R\veps} v_{\eps r}^2 \\
	&=& \frac{2}{\veps} \Feps(\ueps) \veps v_{\eps r}
	- \frac{2(n-1)}{R\veps} v_{\eps r}^2 \\
	&\le& 2\Feps(\ueps) v_{\eps r}
	\qquad \mbox{for all $t\in (0,\tme)$ and } \eps\in [0,1).
  \eas
  Again thanks to (\ref{F}), this implies that
  \bas
	\frac{1}{2} \int_t^{t+\te} \int_{\pO} \frac{1}{\veps} \frac{\partial |\na\veps|^2}{\partial\nu}
	&\le& |\pO| \cdot \int_t^{t+\te} \ueps(R,s) \cdot |v_{\eps r}(R,s)| ds \\
	&\le& |\pO| \cdot \|v_{\eps r}\|_{L^\infty((\frac{R}{2},R)) \times (0,\tme))} 
	\cdot \int_t^{t+\te} \ueps(R,s) ds \\[2mm]
	& & \hs{20mm}
	\qquad \mbox{for all $t\in [0,\tme-\te)$ and } \eps\in [0,1),
  \eas
%  for all $t\in (0,\tme-\te)$ and $\eps\in [0,1)$,
  and that, similarly,
  \bas
	\int_t^{t+\te} \int_{\pO} |\na\veps|^2 \frac{\partial |\na\veps|^2}{\partial\nu} 
	&\le& 2 |\pO| \vst \cdot \|v_{\eps r}\|_{L^\infty((\frac{R}{2},R)) \times (0,\tme))}^3 
	\cdot \int_t^{t+\te} \ueps(R,s) ds \\[2mm]
	& & \hs{20mm}
	\qquad \mbox{for all $t\in [0,\tme-\te)$ and } \eps\in [0,1),
  \eas
  so that since furthermore
  \bas
	\bigg| \int_{\pO} \frac{|\na\veps|^2}{\veps^2} \frac{\partial\veps}{\partial\nu} \bigg| 
	\le \frac{|\pO|}{2\vst^2} \|v_{\eps r}\|_{L^\infty((\frac{R}{2},R)) \times (0,\tme))}^3 
	\qquad \mbox{for all $t\in (0,\tme)$ and } \eps\in [0,1),
  \eas
  the claim results from Lemma \ref{lem9} when combined with Corollary \ref{cor7}.
\end{proof}
\mysection{Energy analysis}\label{sect4}
Our approach toward deriving a suitable relative of (\ref{energy}) is now launched by the following observation.
\begin{lem}\label{lem3}
  Let $\eps\in [0,1)$. Then
  \bea{3.1}
	& & \hs{-16mm}
	\frac{d}{dt} \bigg\{ \io \ueps\ln\ueps + \frac{1}{2} \io \frac{|\na\veps|^2}{\veps} \bigg\}
	+ \io \frac{|\na\ueps|^2}{\ueps}
	+ \io \veps |D^2\ln\veps|^2 \nn\\
	&=&
	-\frac{1}{2} \io \frac{\Feps(\ueps)}{\veps} |\na\veps|^2
	+ \frac{1}{2} \int_{\pO} \frac{1}{\veps} \frac{\partial |\na\veps|^2}{\partial\nu}
	- \frac{1}{2} \int_{\pO} \frac{|\na\veps|^2}{\veps^2} \frac{\partial\veps}{\partial\nu}
	\qquad \mbox{for all $t\in (0,\tme)$.}
  \eea
\end{lem}
\begin{proof}
  According to the no-flux boundary condition accompanying the first equation in (\ref{0eps}),
  \be{3.2}
	\frac{d}{dt} \io \ueps\ln\ueps 
	+ \io \frac{|\na\ueps|^2}{\ueps}
	= \io \Feps'(\ueps) \na\ueps\cdot\na\veps
	\qquad \mbox{for all } t\in (0,\tme),
  \ee
  while on the basis of the second equation in (\ref{0eps}) we first compute
  \bea{3.3}
	\frac{1}{2} \frac{d}{dt} \io \frac{|\na\veps|^2}{\veps}
	&=& \io \frac{1}{\veps} \na\veps\cdot\na \Big\{ \Del\veps - \Feps(\ueps)\veps\Big\}
	- \frac{1}{2} \io \frac{|\na\veps|^2}{\veps^2} \cdot \Big\{ \Del\veps-\Feps(\ueps)\veps\Big\} \nn\\
	&=& \frac{1}{2} \io \frac{1}{\veps} \Del|\na\veps|^2
	- \io \frac{1}{\veps} |D^2\veps|^2
	- \frac{1}{2} \io \frac{|\na\veps|^2}{\veps^2} \Del\veps \nn\\
	& & - \frac{1}{2} \io \frac{\Feps(\ueps)}{\veps} |\na\veps|^2
	- \io \Feps'(\ueps) \na\ueps\cdot\na\veps
	\qquad \mbox{for all } t\in (0,\tme),
  \eea
  because $\na\veps\cdot\na\Del\veps=\frac{1}{2} \Del |\na\veps|^2 - |D^2\veps|^2$.
  Here two integrations by parts show that
  \bas
	\frac{1}{2} \io \frac{1}{\veps} \Del |\na\veps|^2
	= \frac{1}{2} \io \frac{1}{\veps^2} \na\veps\cdot\na|\na\veps|^2
	+ \frac{1}{2} \int_{\pO} \frac{1}{\veps} \frac{\pa |\na\veps|^2}{\pa\nu}
	\qquad \mbox{for all } t\in (0,\tme),
  \eas
  and that
  \bas
	- \frac{1}{2} \io \frac{|\na\veps|^2}{\veps^2} \Del\veps
	= \frac{1}{2} \io \frac{1}{\veps^2} \na\veps\cdot\na |\na\veps|^2
	- \io \frac{|\na\veps|^4}{\veps^3}
	- \frac{1}{2} \int_{\pO} \frac{|\na\veps|^2}{\veps^2} \frac{\pa\veps}{\pa\nu}
	\quad \mbox{for all } t\in (0,\tme),
  \eas
  so that since $\na |\na\veps|^2 = 2D^2\veps\cdot\na\veps$, we obtain that
  \bas
	& & \hs{-20mm}
	\frac{1}{2} \io \frac{1}{\veps} \Del |\na\veps|^2
	- \io \frac{1}{\veps} |D^2\veps|^2
	- \frac{1}{2} \io \frac{|\na\veps|^2}{\veps^2} \Del\veps
	- \frac{1}{2} \int_{\pO} \frac{1}{\veps} \frac{\pa |\na\veps|^2}{\pa\nu}
	+ \frac{1}{2} \int_{\pO} \frac{|\na\veps|^2}{\veps^2} \frac{\pa\veps}{\pa\nu} \nn\\
	&=& - \io \frac{1}{\veps} |D^2\veps|^2
	+ 2\io \frac{1}{\veps^2} \na\veps\cd (D^2\veps\cd\na\veps)
	- \io \frac{|\na\veps|^4}{\veps^3} \\
	&=& - \sum_{i,j=1}^n \io \veps \cd \bigg| \frac{\pa_{x_i x_j} \veps}{\veps} 
		- \frac{\pa_{x_i} \veps \pa_{x_j}\veps}{\veps^2} \bigg|^2 \\
	&=& - \sum_{i,j=1}^n \io \veps \cd \Big| \pa_{x_i x_j} \ln \veps \Big|^2 \\
	&=& - \io \veps |D^2 \ln\veps|^2
	\qquad \mbox{for all } t\in (0,\tme).
  \eas
  Therefore, (\ref{3.3}) is equivalent to (\ref{3.1}).
\end{proof}
In order to make use of the last term on the left of \eqref{3.1}, we will employ the following variant of
a functional inequality which for functions with vanishing normal derivative on $\pO$ has been 
documented in \cite[Lemma~3.3]{win_CPDE12}.
\begin{lem}\label{lem21}
  Let $\vp\in C^2(\bom)$ be such that $\vp>0$ in $\bom$. Then
  \be{21.1}
	\io \frac{|\na\vp|^4}{\vp^3} 
	\le (2+\sqrt{n})^2 \io \vp |D^2\ln \vp|^2
	+ 2 \int_{\pO} \frac{|\na\vp|^2}{\vp^2} \frac{\pa\vp}{\pa\nu}.
  \ee
\end{lem}
\begin{proof}
  We integrate by parts to see that
  \bas
	\io \frac{|\na\vp|^4}{\vp^3}
	&=& \io |\na\ln \vp|^2 \na\ln \vp\cd \na\vp \\
	&=& - \io \vp\na\ln\vp\cd \na |\na\ln\vp|^2
	- \io \vp |\na\ln \vp|^2 \Del \ln \vp
	+ \int_{\pO} \vp |\na\ln \vp|^2 \frac{\pa\ln\vp}{\pa\nu} \\
	&=& -2 \io \frac{1}{\vp} \na\vp \cdot (D^2\ln\vp \cdot\na\vp)
	- \io \frac{|\na\vp|^2}{\vp} \Del\ln\vp
	+ \int_{\pO} \frac{|\na\vp|^2}{\vp^2} \frac{\pa\vp}{\pa\nu}.
  \eas
  As
  \bas
	-2 \io \frac{1}{\vp} \na\vp \cdot (D^2\ln\vp \cdot\na\vp)
	- \io \frac{|\na\vp|^2}{\vp} \Del\ln\vp 
	&\le& (2+\sqrt{n}) \io \frac{|\na\vp|^2}{\vp} |D^2\ln\vp| \\
	&\le& \frac{1}{2} \io \frac{|\na\vp|^4}{\vp^3}
	+ \frac{(2+\sqrt{n})^2}{2} \io \vp |D^2\ln\vp|^2
  \eas
  by Young's inequality, this implies (\ref{21.1}).
\end{proof}
Now an exploitation of the latter in the context of (\ref{3.1}) shows that the boundary regularity features obtained
in Corollary \ref{cor10} imply the following spatially global estimates.
\begin{lem}\label{lem11}
  There exists $C>0$ such that for each $\eps\in [0,1)$, writing $\te=\min\{1,\frac{1}{2}\tme\}$, we have
  \be{11.1}
	\io \ueps(\cdot,t) \ln \ueps(\cdot,t) \le C
	\qquad \mbox{for all $t\in (0,\tme)$}
  \ee
  and
  \be{11.01}
	\io |\na\veps(\cdot,t)|^2 \le C
	\qquad \mbox{for all $t\in (0,\tme)$}
  \ee
  as well as
  \be{11.2}
	\int_t^{t+\te} \io \frac{|\na\ueps|^2}{\ueps} \le C
	\qquad \mbox{for all $t\in [0,\tme-\te)$}
  \ee
  and
  \be{11.4}
	\int_t^{t+\te} \io |\na\veps|^4 \le C
	\qquad \mbox{for all $t\in [0,\tme-\te)$.}
  \ee
\end{lem}
\begin{proof}
  We first employ the Gagliardo-Nirenberg inequality to pick $c_1>0$ such that
  \be{11.5}
	\|\varphi\|_{L^\frac{2(n+2)}{n}(\Om)}^\frac{2(n+2)}{n}
	\le c_1\|\na\varphi\|_{L^2(\Om)}^2 \|\varphi\|_{L^2(\Om)}^\frac{4}{n}
	+ c_1 \|\varphi\|_{L^2(\Om)}^\frac{2(n+2)}{n}
	\qquad \mbox{for all } \varphi\in W^{1,2}(\Om),
  \ee
  and use that $\frac{\xi\ln\xi}{\xi^\frac{n+2}{n}} \to 0$ as $\xi\to +\infty$ in choosing $c_2>0$ such that abbreviating
  $c_3:=\io u^{(0)}$ we have
  \be{11.6}
	\xi\ln\xi \le \frac{2}{c_1 c_3^\frac{2}{n}} \xi^\frac{n+2}{n} + c_2
	\qquad \mbox{for all } \xi>0.
  \ee
  Then writing $c_4:=\frac{1}{(2+\sqrt{n})^2}$ and $c_5:=\frac{|\Om|}{8c_4} \|v^{(0)}\|_{L^\infty(\Om)}$,
  by means of (\ref{11.6}), Young's inequality, (\ref{11.5}), (\ref{mass}) and (\ref{vinfty}) we see that for each $\eps\in [0,1)$,  
  \bas
	\yeps(t):=\io \ueps(\cdot,t)\ln\ueps(\cdot,t) + \frac{1}{2} \io \frac{|\na\veps(\cdot,t)|^2}{\veps(\cdot,t)},
	\qquad t\in [0,\tme),
  \eas
  has the property that
  \bas
	\yeps(t)
	&\le& \frac{2}{c_1 c_3^\frac{2}{n}} \io \ueps^\frac{n+2}{n} + c_2 |\Om|
	+ \frac{c_4}{2} \io \frac{|\na\veps|^4}{\veps^3} 
	+ \frac{1}{8c_4} \io \veps \\
	&=& \frac{2}{c_1 c_3^\frac{2}{n}} \|\sqrt{\ueps}\|_{L^\frac{2(n+2)}{n}(\Om)}^\frac{2(n+2)}{n} + c_2|\Om|
	+ \frac{c_4}{2} \io \frac{|\na\veps|^4}{\veps^3} 
	+ \frac{1}{8c_4} \io \veps \\
	&\le& \frac{2}{c_3^\frac{2}{n}} \|\na\sqrt{\ueps}\|_{L^2(\Om)}^2 \|\sqrt{\ueps}\|_{L^2(\Om)}^\frac{4}{n}
	+ \frac{2}{c_3^\frac{2}{n}} \|\sqrt{\ueps}\|_{L^2(\Om)}^\frac{2(n+2)}{n}
	+ c_2|\Om|
	+ \frac{c_4}{2} \io \frac{|\na\veps|^4}{\veps^3} 
	+ c_5 \\
	&=& \frac{1}{2} \io \frac{|\na\ueps|^2}{\ueps} 
	+ 2c_3 + c_2|\Om| 
	+ \frac{c_4}{2} \io \frac{|\na\veps|^4}{\veps^3} 
	+ c_5
	\qquad \mbox{for all } t\in (0,\tme),
  \eas
  so that since Lemma \ref{lem21} warrants that
  \bas
	c_4 \io \frac{|\na\veps|^4}{\veps^3}
	\le \io \veps |D^2\ln\veps|^2
	+ 2c_4 \int_{\pO} \frac{|\na\veps|^2}{\veps^2} \frac{\pa\veps}{\pa\nu}
	\qquad \mbox{for all } t\in (0,\tme),
  \eas
  it follows that
  \bea{11.7}
	\hs{-6mm}
	\yeps(t) 
	+ \frac{1}{2} \io \frac{|\na\ueps|^2}{\ueps}
	+ \frac{c_4}{2} \io \frac{|\na\veps|^4}{\veps^3}
	&\le& \io \frac{|\na\ueps|^2}{\ueps}
	+ \io \veps |D^2\ln\veps|^2 \nn\\
	& & + 2c_4 \int_{\pO} \frac{|\na\veps|^2}{\veps^2} \frac{\pa\veps}{\pa\nu}
	+ c_6
	\qquad \mbox{for all } t\in (0,\tme),
  \eea
  with $c_6:=2c_3 + c_2|\Om|+c_5$.
  Accordingly, from (\ref{3.1}) we infer upon dropping a favorably signed summand therein that
  \bea{11.8}
	& & \hs{-17mm}
	\yeps'(t)+\yeps(t)
	+\frac{1}{2} \io \frac{|\na\ueps|^2}{\ueps}	
	+ \frac{c_4}{2} \io \frac{|\na\veps|^4}{\veps^3} \nn\\
	& & \hs{-5mm}
	\le \heps(t)
	:= \frac{1}{2} \int_{\pO} \frac{1}{\veps} \frac{\pa |\na\veps|^2}{\pa\nu}
	+ \Big(2c_4-\frac{1}{2}\Big) \int_{\pO} \frac{|\na\veps|^2}{\veps^2} \frac{\pa\veps}{\pa\nu}
	+ c_6
	\qquad \mbox{for all } t\in (0,\tme).
  \eea
  As Corollary \ref{cor10} provides $c_7>0$ such that
  \be{11.9}
	\int_t^{t+\te} \heps(s) ds \le c_7
	\qquad \mbox{for all $t\in [0,\tme-\te)$ and } \eps\in [0,1),
  \ee
  through Lemma \ref{lem77} this firstly ensures that for all $t\in (0,\tme)$ and $\eps\in [0,1)$,
  \bea{11.10}
	\yeps(t)
	&\le& \yeps(0) e^{-t} + \int_0^t e^{-(t-s)} \heps(s) ds \nn\\
	&\le& |\yeps(0)| + \frac{c_7\te}{1-e^{-\te}} \nn\\
	&\le& c_8:= \io u^{(0)}|\ln u^{(0)}| + \frac{1}{2} \io \frac{|\na v^{(0)}|^2}{v^{(0)}}
	+ \frac{c_7}{1-e^{-1}},
%	\qquad \mbox{for all $t\in (0,\tme)$ and } \eps\in [0,1),
  \eea
  again because $\frac{\te}{1-e^{-\te}} \le \frac{1}{1-e^{-1}}$ for all $\eps\in [0,1)$.
  Going back to (\ref{11.8}), from this we thereupon infer that
  \bea{11.11}
	& &  \hs{-30mm}
	\frac{1}{2} \int_t^{t+\te} \io \frac{|\na\ueps|^2}{\ueps}
	+ \frac{c_4}{2} \int_t^{t+\te} \io \frac{|\na\veps|^4}{\veps^3} \nn\\
	&\le& \yeps(t) - \yeps(t+\te)
	- \int_t^{t+\te} \yeps(s) ds 
	+ \int_t^{t+\te} \heps(s) ds \nn\\
	&\le& c_8 + \frac{2|\Om|}{e} + c_7
	\qquad \mbox{for all $t\in [0,\tme-\te)$ and } \eps\in [0,1),
  \eea
  since evidently 
  \be{11.12}
	\yeps(t) \ge \io \ueps\ln \ueps \ge - \frac{|\Om|}{e}
	\qquad \mbox{for all $t\in [0,\tme)$ and } \eps\in [0,1).
  \ee
  Once more relying on (\ref{vinfty}), from (\ref{11.11}) we obtain both (\ref{11.2}) and (\ref{11.4}), whereas (\ref{11.1})
  and (\ref{11.01}) similarly result from (\ref{11.10}) due to the second inequality in (\ref{11.12}).
\end{proof}
\mysection{The two-dimensional case. Proof of Theorem \ref{theo16}}
In this section we concentrate on the two-dimensional setting of Theorem~\ref{theo16}. 
Since the solutions there will already turn out to be bounded and classical, it is not necessary to resort to an approximation by means of \eqref{0eps} for $ε>0$. Throughout this section, we will therefore directly address the solutions $(u,v):=(u_0,v_0)$ of \eqref{0eps} 
obtained for $\eps=0$.\abs
Based on the information provided by Lemma \ref{lem11}, we can combine the outcomes of two further testing procedures applied 
to (\ref{0eps}) in quite a standard manner, and thereby achieve the following key toward higher order bounds:
\begin{lem}\label{lem14}
  Let $n=2$. Then there exists $C>0$ such that the solution $(u,v)\equiv (u_0,v_0)$ of (\ref{0eps}), as corresponding
  to the choice $\eps=0$, satisfies
  \be{14.1}
	\io |\na v(\cdot,t)|^4 \le C
	\qquad \mbox{for all } t\in (0,\tm),
  \ee
  where $\tm:=T_{max,0}$ is as accordingly provided by Lemma \ref{lem1}.
\end{lem}
\begin{proof}
  On the basis of (\ref{0eps}) when restricted to $\eps=0$, by means of Young's inequality we see that
  \be{14.2}
	\frac{d}{dt} \io u^2 + \io |\na u|^2
	= - \io |\na u|^2
	+ 2\io u\na u\cdot\na v
	\le \io u^2 |\na v|^2
	\qquad \mbox{for all } t\in (0,\tm),
  \ee
  and that for all $t\in (0,\tm)$,
  \bea{14.3}
	\frac{1}{4} \frac{d}{dt} \io |\na v|^4
	&=& \io |\na v|^2 \na v\cdot\na \big\{ \Del v - uv \big\} \nn\\
	&=& \frac{1}{2} \io |\na v|^2 \Del |\na v|^2
	- \io |\na v|^2 |D^2 v|^2
	- \io u |\na v|^4
	- \io v |\na v|^2 \na u\cdot\na v \nn\\
	&=& - \frac{1}{2} \io \Big|\na |\na v|^2 \Big|^2
	+ \frac{1}{2} \int_{\pO} |\na v|^2 \frac{\pa |\na v|^2}{\pa\nu} \nn\\
	& & - \io |\na v|^2 |D^2 v|^2
	- \io u |\na v|^4
	- \io v |\na v|^2 \na u\cdot\na v \nn\\
	&\le& - \frac{1}{2} \io \Big| \na |\na v|^2 \Big|^2
	+ \frac{1}{2} \int_{\pO} |\na v|^2 \frac{\pa |\na v|^2}{\pa\nu}
	- \io v |\na v|^2 \na u\cdot\na v \nn\\
	&\le& - \frac{1}{2} \io \Big| \na |\na v|^2 \Big|^2
	+ \frac{1}{2} \int_{\pO} |\na v|^2 \frac{\pa |\na v|^2}{\pa\nu}
	+ \|v_0\|_{L^\infty(\Om)} \io |\na u| \cd |\na v|^3
%	\qquad \mbox{for all } t\in (0,\tm)
  \eea
  because of (\ref{vinfty}).
  To proceed from this, we employ the Gagliardo-Nirenberg inequality to find $c_1>0$ such that
  \bea{14.4}
	\io |\na v|^6
	&=& \Big\| |\na v|^2 \Big\|_{L^3(\Om)}^3 \nn\\
	&\le& c_1 \Big\| \na |\na v|^2 \Big\|_{L^2(\Om)}^2 
	\Big\| |\na v|^2 \Big\|_{L^1(\Om)}
	+ c_1 \Big\| |\na v|^2 \Big\|_{L^1(\Om)}^3 \nn\\
	&\le& c_1 c_2 \Big\| \na |\na v|^2 \Big\|_{L^2(\Om)}^2 
	+ c_1 c_2^3
	\qquad \mbox{for all } t\in (0,\tm),
  \eea
  with finiteness of $c_2:=\sup_{t\in (0,\tm)} \io |\na v(\cdot,t)|^2$ being asserted by Lemma \ref{lem11}.
  We then fix $a>0$ suitably small such that
  \be{14.5}
	8a^2 \|v^{(0)}\|_{L^\infty(\Om)}^2
	\le \frac{a}{c_1 c_2},
  \ee
  and take $\eta>0$ small enough fulfilling
  \be{14.6}
	\frac{1}{2\eta} \ge 1 + \Big( \frac{2c_1 c_2}{a}\Big)^\frac{1}{2},
  \ee
  whereupon an application of a well-known variant of the Gagliardo-Nirenberg inequality (\cite{biler_hebisch_nadzieja})
  shows that since Lemma \ref{lem11} warrants boundedness of $(u(\cdot,t)\ln u(\cdot,t))_{t\in (0,\tm)}$ in $L^1(\Om)$,
  there exists $c_3>0$ such that
  \be{14.7}
	\io u^3 \le \eta \io |\na u|^2 + c_3
	\qquad \mbox{for all } t\in (0,\tm).
  \ee
  We now let
  \bas
	y(t):=\io u^2(\cdot,t) + a \io |\na v(\cdot,t)|^4,
	\qquad t\in [0,\tm),
  \eas
  and combine (\ref{14.2}) with (\ref{14.3}) to obtain that since
  \bas
	\io |\na u|^2 \ge \frac{1}{2} \io |\na u|^2 + \frac{1}{2\eta} \io u^3 - \frac{c_3}{2\eta}
	\qquad \mbox{for all } t\in (0,\tm)
  \eas
  and
  \bas
	2 \io \Big|\na |\na v|^2 \Big|^2
	\ge \frac{2}{c_1c_2} \io |\na v|^6 - 2c_2^2
	\qquad \mbox{for all } t\in (0,\tm)
  \eas
  by (\ref{14.7}) and (\ref{14.4}), we have
  \bea{14.8}
	& & \hs{-20mm}
	y'(t) + y(t)
	+ \frac{1}{2} \io |\na u|^2
	+ \frac{1}{2\eta} \io u^3
	+ \frac{2a}{c_1c_2} \io |\na v|^6 \nn\\
	&\le& \io u^2 + a\io |\na v|^4 \nn\\
	& & + \frac{c_3}{2\eta} + 2ac_2^2 \nn\\
	& & + \io u^2 |\na v|^2 \nn\\
	& & + 4a \|v^{(0)}\|_{L^\infty(\Om)} \io |\na u| \cdot |\na v|^3 \nn\\
	& & + 2a \int_{\pO} |\na v|^2 \frac{\pa |\na v|^2}{\pa\nu}
	\qquad \mbox{for all } t\in (0,\tm).
  \eea
  Here due to Young's inequality,
  \bas
	& & \hs{-20mm}
	\io u^2
	+ a \io |\na v|^4
	+ \io u^2 |\na v|^2
	+ 4a \|v^{(0)}\|_{L^\infty(\Om)} \io |\na u|\cd |\na v|^3 \nn\\
	&\le& \io u^2 + a\io |\na v|^4 + \io u^2 |\na v|^2
	+ \frac{1}{2} \io |\na u|^2
	+ 8a^2 \|v^{(0)}\|_{L^\infty(\Om)}^2 \io |\na v|^6 \\
	&=& \io u^2 \\
	& & + \io \Big\{ \frac{a}{2c_1 c_2} |\na v|^6 \Big\}^\frac{2}{3} \cdot (2c_1 c_2)^\frac{2}{3} a^\frac{1}{3} \\
	& & + \io \Big\{ \frac{a}{2c_1 c_2} |\na v|^6 \Big\}^\frac{1}{3} \cdot \Big(\frac{2c_1c_2}{a}\Big)^\frac{1}{3} u^2 \\
	& & 
	+ \frac{1}{2} \io |\na u|^2
	+ 8a^2 \|v^{(0)}\|_{L^\infty(\Om)}^2 \io |\na v|^6 \\
	&\le& \io u^3 + |\Om| \\
	& & 
	+ \frac{a}{2c_1c_2} \io |\na v|^6
	+ (2c_1c_2)^2 a|\Om| \\
	& & + \frac{a}{2c_1c_2} \io |\na v|^6
	+ \Big(\frac{2c_1 c_2}{a}\Big)^\frac{1}{2} \io u^3 \\
	& & + \frac{1}{2} \io |\na u|^2
	+ 8a^2 \|v^{(0)}\|_{L^\infty(\Om)}^2 \io |\na v|^6 \\
	&=& \Big\{ 1 + \Big(\frac{2c_1c_2}{a}\Big)^\frac{1}{2}\Big\} \cdot \io u^3
	+ \Big\{ \frac{a}{c_1c_2} + 8a^2 \|v^{(0)}\|_{L^\infty(\Om)}^2 \Big\} \cdot \io |\na v|^6 \\
	& & + |\Om| + (2c_1c_2)^2 a|\Om| \\
	& & + \frac{1}{2} \io |\na u|^2
	\qquad \mbox{for all } t\in (0,\tm),
  \eas
  whence drawing on (\ref{14.6}) and (\ref{14.5}) we infer from (\ref{14.8}) that
  \bas
	y'(t) + y(t)
	\le h(t):=\frac{c_3}{2\eta} + 2ac_2^2 + |\Om| + (2c_1c_2)^2 a|\Om|
	+ 2a\int_{\pO} |\na v|^2 \frac{\pa |\na v|^2}{\pa\nu}
	\qquad \mbox{for all } t\in (0,\tm).
  \eas
  Since Corollary \ref{cor10} ensures that $\sup_{t\in (0,\tm-\tau_0)} \int_t^{t+\tau_0} h(s)ds$, with $τ_0=\min\{1,\frac12\tm\}$, is finite, by means of
  Lemma \ref{lem77} this entails that $y$ is bounded in $(0,\tm)$, which in particular implies (\ref{14.1})
  with some suitably large $C>0$.
\end{proof}
Indeed, this implies boundedness in the respective first solution components.
\begin{lem}\label{lem15}
  Let $n=2$. Then there exists $C>0$ such that with $(u,v)\equiv (u_0,v_0)$ and $\tm=T_{max,0}$ taken from Lemma \ref{lem1} 
  we have
  \be{15.1}
	\|u(\cdot,t)\|_{L^\infty(\Om)} \le C
	\qquad \mbox{for all } t\in (0,\tm).
  \ee
\end{lem}
\begin{proof}
  We write the first equation in (\ref{0eps}) for $\eps=0$ in the form
  $u_t = \Delta u + \na\cdot (b(x,t) u)$, $(x,t)\in\Om\times (0,\tm)$, and note that according to Lemma \ref{lem14},
  $b:=-\na v$ belongs to $L^\infty((0,\tm);L^q(\Om))$ with $q:=4$ exceeding the considered spatial dimension.
  Since $(\na u+b(x,t)u)\cdot\nu=0$ on $\pO\times (0,\tm)$, we may therefore refer to a boundedness statement derived by means
  of a straightforward Moser-type iteration (\cite{ding_win}) to directly obtain (\ref{15.1}).
\end{proof}
For the proof of Theorem~\ref{theo16}, we are merely lacking a 
transfer of the boundedness properties we have just obtained 
to the spaces that actually occur in the extensibility criterion \eqref{ext}:
\begin{proof}[Proof of Theorem \ref{theo16}]
  Based on the outcome of Lemma \ref{lem15}, we may again utilize known smoothing properties of the Dirichlet heat semigroup
  on $\Omega$, and additionally employ a standard result on gradient H\"older 
  regularity in scalar parabolic equations (\cite{lieberman}),
  to find $c_1>0, c_2>0$ and $\theta_1\in (0,1)$ such that
  \be{16.2}
	\|v(\cdot,t)\|_{W^{1,\infty}(\Om)} \le c_1
	\qquad \mbox{for all } t\in (0,\tm)
  \ee
  and
  \bas
	\|\na v\|_{C^{\theta_1,\frac{\theta_1}{2}}(\bom\times [t,t+\tau_0])} \le c_2
	\qquad \mbox{for all $t\in (\frac{\tau_0}{4},\tm-\tau_0)$,}
  \eas
  where again $\tau_0=\min\{1,\frac{1}{2}\tm\}$.
  According to (\ref{16.2}), we may thereafter rely on the latter token once again to infer from Lemma \ref{lem15} and the first
  sub-problem in (\ref{0eps}) that with some $c_3>0$ and $\theta_2\in (0,1)$ we have
  \bas
	\|u\|_{C^{1+\theta_2,\theta_2}(\bom\times [t,t+\tau_0])} \le c_3
	\qquad \mbox{for all $t\in (\frac{\tau_0}{2},\tm-\tau_0)$,}
  \eas
  which combined with (\ref{16.2}) and (\ref{ext}) shows that Lemma \ref{lem1} indeed asserts that $\tm=\infty$, whereupon
  (\ref{16.1}) becomes a consequence of Lemma \ref{lem15} and (\ref{16.2}).
\end{proof}
\mysection{The case \texorpdfstring{$n\ge 3$}{n > 2}. Proof of Theorem \ref{theo25}}
The solution concept to be pursued in higher-dimensional cases appears to be quite natural.
\begin{defi}\label{dw}
  Let
  \be{w1}
	\left\{ \begin{array}{l}
	u\in L^1_{loc}([0,\infty);W^{1,1}(\Om)) \qquad \mbox{and} \\[1mm]
	v\in L^1_{loc}([0,\infty);W^{1,1}(\Om))
	\end{array} \right.
  \ee
  be nonnegative and such that $v(\cdot,t)-\vst\in W_0^{1,1}(\Om)$ for a.e.~$t>0$, and that
  \be{w2}
	u\na v\in L^1_{loc}(\bom\times [0,\infty);\R^n)
	\qquad \mbox{and} \qquad
	uv\in L^1_{loc}(\bom\times [0,\infty)).
  \ee
  Then $(u,v)$ will be called a {\em global weak solution of (\ref{0})} if
  \be{wu}
	-\int_0^\infty \io u\vp_t - \io u^{(0)}\vp(\cdot,0)
	= - \int_0^\infty \io \na u\cdot\na\vp
	+ \int_0^\infty \io u\na v\cdot\na\vp
  \ee
  for all $\vp\in C_0^\infty(\bom\times [0,\infty))$, and if
  \be{wv}
	-\int_0^\infty \io v\vp_t - \io v^{(0)}\vp(\cdot,0)
	= - \int_0^\infty \io \na v\cdot\na\vp
	- \int_0^\infty uv\vp
  \ee
  for all $\vp\in C_0^\infty(\Om\times [0,\infty))$.
\end{defi}
In order to construct such solutions, we now utilize solutions of the approximate versions of \eqref{0eps}, that is, those
corresponding to positive values of $\eps$. 
As can easily be seen, the strength of the accordingly regularizing features $F_\eps$ is sufficient to ensure
that each of these solutions is global in time: 
\begin{lem}\label{lem13}
  Let $n\ge 3$ and $\eps\in (0,1)$. Then $\tme=\infty$.
\end{lem}
\begin{proof}
  Supposing for contradiction that $\tme$ be finite for some $\eps\in (0,1)$, we note that since $0\le\Feps\le\frac{1}{\eps}$
  and hence $\Feps(\ueps)\veps$ belongs to $L^\infty(\Om\times (0,\tme))$ by (\ref{vinfty}), the standard result on
  parabolic $C^{1+\theta,\theta}$ regularity from \cite{lieberman} would become applicable so as to ensure that
  \be{13.1}
	\veps\in C^{1+\theta_1,\theta_1}(\bom\times \mbox{$[\frac{1}{4}\tme,\tme]$})
  \ee
  for some $\theta_1\in (0,1)$.
  As $0 \le \xi\Feps'(\xi) \le \frac{1}{\eps}$ for all $\xi\ge 0$, this would especially assert boundedness of
  $\ueps\Feps'(\ueps)\na\veps$ in $\Om\times (\frac{1}{4}\tme,\tme)$ and hence, firstly, warrant the inclusion
  $\ueps\in L^\infty(\Om\times (\frac{1}{4}\tme,\tme))$ through the Moser iteration result from \cite{ding_win}.
  Thereafter, another application of standard parabolic regularity theory (\cite{lieberman}) would entail that, in fact,
  $\ueps\in C^{1+\theta_2,\theta_2}(\bom\times [\frac{1}{2}\tme,\tme])$ for some $\theta_2\in (0,1)$, which together with
  (\ref{13.1}) would clearly contradict (\ref{ext}).
\end{proof}
In passing to the limit $\eps\searrow0$, we will use the following consequences of Lemma \ref{lem11} on further
spatio-temporal bounds.
\begin{lem}\label{lem22}
  Let $n\ge 3$. Then there exists $C>0$ such that
  \be{22.1}
	\int_t^{t+1} \io \ueps^\frac{n+2}{n} \le C
	\qquad \mbox{for all $t\ge 0$ and } \eps\in (0,1)
  \ee
  and
  \be{22.2}
	\int_t^{t+1} \io |\na\ueps|^\frac{n+2}{n+1} \le C
	\qquad \mbox{for all $t\ge 0$ and } \eps\in (0,1).
  \ee
\end{lem}
\begin{proof}
  From the Gagliardo-Nirenberg inequality and (\ref{mass}) we obtain $c_1>0$ and $c_2>0$ such that
  \bas
	\io \ueps^\frac{n+2}{n}
	= \|\sqrt{\ueps}\|_{L^\frac{2(n+2)}{n}(\Om)}^{\frac{2(n+2)}n}
	&\le& c_1 \|\na\sqrt{\ueps}\|_{L^2(\Om)}^2 \|\sqrt{\ueps}\|_{L^2(\Om)}^\frac{4}{n}
	+ c_1 \|\sqrt{\ueps}\|_{L^2(\Om)}^\frac{2(n+2)}{n} \\
	&\le& c_2 \io \frac{|\na\ueps|^2}{\ueps} + c_2
	\qquad \mbox{for all $t> 0$ and } \eps\in (0,1),
  \eas
  and an application of Young's inequality shows that
  \bas
	\io |\na\ueps|^\frac{n+2}{n+1}
	&=& \io \Big\{ \frac{|\na\ueps|^2}{\ueps} \Big\}^\frac{n+2}{2(n+1)} \cdot \ueps^\frac{n+2}{2(n+1)} \\
	&\le& \io \frac{|\na\ueps|^2}{\ueps} + \io \ueps^\frac{n+2}{n}
	\qquad \mbox{for all $t> 0$ and } \eps\in (0,1).
  \eas
  Therefore, both (\ref{22.1}) and (\ref{22.2}) result from (\ref{11.2}).
\end{proof}
In preparation of an Aubin--Lions type argument, we additionally require some weak estimates for time derivatives.
\begin{lem}\label{lem23}
  Let $n\in \{3,4,5\}$. Then for all $T>0$ one can find $C(T)>0$ such that
  \be{23.1}
	\int_0^T \|u_{\eps t}(\cdot,t)\|_{(W^{1,1+\frac{5n+2}{6-n}}(\Om))^\star}^{1+\frac{6-n}{5n+2}} dt \le C(T)
	\qquad \mbox{for all } \eps\in (0,1)
  \ee
  and
  \be{23.2}
	\int_0^T \|v_{\eps t}(\cdot,t)\|_{(W_0^{1,1+\frac{(n+1)(n-2)}{3n+2}}(\Om))^\star}^\frac{n+2}{n} \le C(T)
	\qquad \mbox{for all } \eps\in (0,1).
  \ee
\end{lem}
\begin{proof}
  We note that since $2\le n<6$, we have $1+\frac{5n+2}{6-n}=\frac{4(n+2)}{6-n}\ge n+2$, so that we can fix $c_1>0$ with the property
  that $\|\na\psi\|_{L^{n+2}(\Om)} + \|\na\psi\|_{L^\frac{4(n+2)}{6-n}(\Om)} \le c_1$ for all $\psi\in C^1(\bom)$ such that
  $\|\psi\|_{W^{1,1+\frac{5n+2}{6-n}}(\Om)} \le 1$.
  Given any such $\psi$, using (\ref{0eps}) along with the H\"older inequality, we thus obtain that since $|\Feps'|\le 1$ for 
  all $\eps\in (0,1)$ by (\ref{F}),
  \bas
	\bigg| \io u_{\eps t} \psi \bigg|
	&=& \bigg| - \io \na\ueps\cdot\na\psi
	+ \io \ueps\Feps'(\ueps) \na\veps\cdot\na\psi \bigg| \\
	&\le& \|\na\ueps\|_{L^\frac{n+2}{n+1}(\Om)} \|\na\psi\|_{L^{n+2}(\Om)}
	+ \|\ueps\|_{L^\frac{n+2}{n}(\Om)} \|\na\veps\|_{L^4(\Om)} \|\na\psi\|_{L^\frac{4(n+2)}{6-n}(\Om)} \\
	&\le& c_1 \|\na\ueps\|_{L^\frac{n+2}{n+1}(\Om)}
	+ c_1 \|\ueps\|_{L^\frac{n+2}{n}(\Om)} \|\na\veps\|_{L^4(\Om)}
	\qquad \mbox{for all $t>0$ and } \eps\in (0,1).
  \eas
  For all $t>0$ and $\eps\in (0,1)$, we therefore have
  \bea{23.3}
	\|u_{\eps t}\|_{(W^{1,1+\frac{5n+2}{6-n}}(\Om))^\star}^{1+\frac{6-n}{5n+2}}
	&\le& (2c_1)^\frac{4n+8}{5n+2} \cdot \bigg\{ \|\na\ueps\|_{L^\frac{n+2}{n+1}(\Om)}^\frac{4n+8}{5n+2}
	+ \|\ueps\|_{L^\frac{n+2}{n}(\Om)}^\frac{4n+8}{5n+2} \|\na\veps\|_{L^4(\Om)}^\frac{4n+8}{5n+2} \bigg\} \nn\\
	&\le& (2c_1)^\frac{4n+8}{5n+2} \cdot \bigg\{ \|\na\ueps\|_{L^\frac{n+2}{n+1}(\Om)}^\frac{n+2}{n+1}
	+ 1
	+ \|\ueps\|_{L^\frac{n+2}{n}(\Om)}^\frac{n+2}{n} 
	+ \|\na\veps\|_{L^4(\Om)}^4 \bigg\}
%	\qquad \mbox{for all $t>0$ and } \eps\in (0,1)
  \eea
  by Young's inequality, because $1+\frac{6-n}{5n+2}=\frac{4n+8}{5n+2}\le \frac{n+2}{n+1}$, again due to the fact that $n\ge 2$.
  Integrating (\ref{23.3}) in time shows that (\ref{23.1}) is implied by Lemma \ref{lem11} and Lemma \ref{lem22}.\abs
  Likewise, to derive (\ref{23.2}) we observe that since
  the inequality $n\ge 3$ warrants that $1+\frac{(n+1)(n-2)}{3n+2}=\frac{n(n+2)}{3n+2} \ge \frac{4}{3}$, and since
  $W^{1,\frac{n(n+2)}{3n+2}}(\Om) \hra L^\frac{n+2}{2}(\Om)$ due to the fact that $1-\frac{3n+2}{n+2}=-\frac{2n}{n+2}$, 
  with some $c_2>0$ we have
  $\|\na\psi\|_{L^\frac{4}{3}(\Om)} + \|\psi\|_{L^\frac{n+2}{2}(\Om)} \le c_1$ for all 
  $\psi\in {\cal B} :=\Big\{ \wh{\psi} \in C_0^\infty(\Om) \ \Big| \ \|\wh{\psi}\|_{W^{1,1+\frac{(n+1)(n-2)}{3n+2}}(\Om)} \le 1\Big\}$.
  Therefore, the second equation in (\ref{0eps}) shows that due to the H\"older inequality, (\ref{F}) and Young's inequality,
  \bas
	\|v_{\eps t}\|_{W_0^{1,1+\frac{(n+1)(n-2)}{3n+2}}(\Om))^\star}^\frac{n+2}{n}
	&=& \sup_{\psi\in {\cal B}}
		%\sup_{\psi\in C_0^\infty(\Om), \ \|\psi\|_{W^{1,\frac{n(n+2)}{3n+2}}(\Om)} \le 1} 
		\bigg| \io v_{\eps t} \psi \bigg|^\frac{n+2}{n} \\
	&=& \sup_{\psi\in {\cal B}}
		%\sup_{\psi\in C_0^\infty(\Om), \ \|\psi\|_{W^{1,\frac{n(n+2)}{3n+2}}(\Om)} \le 1} 
	\bigg| - \io \na\veps\cdot\na\psi - \io \Feps(\ueps)\veps \psi \bigg|^\frac{n+2}{n} \\
	&\le& \sup_{\psi\in {\cal B}}
		%\sup_{\psi\in C_0^\infty(\Om), \ \|\psi\|_{W^{1,\frac{n(n+2)}{3n+2}}(\Om)} \le 1} 
	\bigg\{ \|\na\veps\|_{L^4(\Om)} \|\na\psi\|_{L^\frac{4}{3}(\Om)}
	+ \|\ueps\|_{L^\frac{n+2}{n}(\Om)} \|\veps\|_{L^\infty(\Om)} \|\psi\|_{L^\frac{n+2}{2}(\Om)} \bigg\}^\frac{n+2}{n} \\
	&\le& c_2^\frac{n+2}{n} \cdot 
	\bigg\{ \|\na\veps\|_{L^4(\Om)} 
	+ \|\ueps\|_{L^\frac{n+2}{n}(\Om)} \|\veps\|_{L^\infty(\Om)} \bigg\}^\frac{n+2}{n} \\
	&\le& (2c_2)^\frac{n+2}{n} \cdot 
	\bigg\{ \|\na\veps\|_{L^4(\Om)}^\frac{n+2}{n}
	+ \|\ueps\|_{L^\frac{n+2}{n}(\Om)}^\frac{n+2}{n} \|\veps\|_{L^\infty(\Om)}^\frac{n+2}{n} \bigg\} \\
	&\le& (2c_2)^\frac{n+2}{n} \cdot 
	\bigg\{ \|\na\veps\|_{L^4(\Om)}^4 + 1
	+ \|\ueps\|_{L^\frac{n+2}{n}(\Om)}^\frac{n+2}{n} \|\veps\|_{L^\infty(\Om)}^\frac{n+2}{n} \bigg\} 
%	\qquad \mbox{for all $t>0$ and } \eps\in (0,1),
  \eas
  for all $t>0$ and $\eps\in (0,1)$, 
  as clearly $\frac{n+2}{n} \le 4$. In view of Lemma \ref{lem11}, Lemma \ref{lem22} and (\ref{vinfty}), the inequality in (\ref{23.2})
  thus results upon an integration.
\end{proof}
Our limit passage has thereby been prepared:
\begin{lem}\label{lem24}
  Let $n\in \{3,4,5\}$. Then there exist $(\eps_j)_{j\in\N}\subset (0,1)$, fulfilling $\eps_j\searrow 0$ as $j\to\infty$,
  as well as nonnegative functions
  $u$ and $v$ on $\Om\times (0,\infty)$ which satisfy (\ref{25.01}), for which $(u(\cdot,t),v(\cdot,t))$ is radially
  symmetric for a.e.~$t>0$, for which as $\eps=\eps_j\searrow 0$ we have
  %\begin{eqnarray}
  \begin{align}
	 \ueps&\to u
	 &&\mbox{in } \bigcap_{p\in [1,\frac{n+2}{n})} L^p_{loc}(\bom\times [0,\infty)) \mbox{ and a.e.~in } \Om\times (0,\infty),
	\label{24.2} \\
		 \ueps&\wto u
	 &&\mbox{in } L^{\frac{n+2}n}_{loc}(\bom\times [0,\infty)),
	\label{u-weak-1plus2overn} \\
	\na\ueps&\wto\na u
	 &&\mbox{in } L^\frac{n+2}{n+1}_{loc}(\bom\times [0,\infty)),
	\label{24.3} \\
	\veps&\to v
	 &&\mbox{in } \bigcap_{p\in [1,\infty)} L^p_{loc}(\bom\times [0,\infty)) \mbox{ and a.e.~in } \Om\times (0,\infty)
	\qquad \mbox{and} 
	\label{24.4} \\
	\na\veps&\wto\na v
	 &&\mbox{in } L^4_{loc}(\bom\times [0,\infty)),
	\label{24.5} \\
	\na\veps&\wsto\na v
	 &&\mbox{in } L^\infty_{loc}((0,\infty);L^2(\Om)),
	\label{24.5b} 
	\end{align}
 % \eea
 %\end{eqnarray}
  and such that $(u,v)$ is a global weak solution of (\ref{0}) in the sense of Definition \ref{dw}.
\end{lem}
\begin{proof}
  Given $T>0$, from Lemma \ref{lem22} and Lemma \ref{lem23} we know that
  \bas
	(\ueps)_{\eps\in (0,1)} \mbox{ is bounded in } L^\frac{n+2}{n+1}\big( (0,T); W^{1,\frac{n+2}{n+1}}(\Om)\big)
  \eas
  and that
  \bas
	(u_{\eps t})_{\eps\in (0,1)} \mbox{ is bounded in } 
		L^{1+\frac{6-n}{5n+2}} \Big( (0,T); \big(W^{1,1+\frac{5n+2}{6-n}}(\Om)\big)^\star \Big),
  \eas
  while according to Lemma \ref{lem11}, (\ref{vinfty}) and Lemma \ref{lem23},
  \bas
	(\veps-\vst)_{\eps\in (0,1)} \mbox{ is bounded in } L^4\big( (0,T); W^{1,4}(\Om)\big)
  \eas
  and
  \bas
	\big(\partial_t (\veps-\vst)\big)_{\eps\in (0,1)} \mbox{ is bounded in } 
		L^\frac{n+2}{n} \Big( (0,T); \big(W_0^{1,1+\frac{(n+1)(n-2)}{3n+2}}(\Om)\big)^\star \Big),
  \eas
  because clearly $\partial_t (\veps-\vst) = v_{\eps t}$.
  Therefore, two applications of an Aubin-Lions lemma (\cite{simon}) provide $(\eps_j)_{j\in\N} \subset (0,1)$
  as well as nonnegative radially symmetric functions $u\in L^\frac{n+2}{n+1}_{loc}([0,\infty);W^{1,\frac{n+2}{n+1}}(\Om))$
  and $(v-\vst)\in L^4_{loc}([0,\infty);W_0^{1,4}(\Om))$ such that $ \eps_j\searrow 0$ as $j\to\infty$, that 
  (\ref{24.3}) and (\ref{24.5}) hold, and that
  $(\ueps,\veps)\to (u,v)$ a.e.~in $\Omega\times (0,\infty)$ as
  $\eps=\eps_j\searrow 0$.
  Since furthermore $(\ueps)_{\eps\in (0,1)}$ is bounded in $L^\infty((0,\infty);L^1(\Om))$ and in 
  $L^\frac{n+2}{n}(\Om\times (0,T))$ for all $T>0$ by (\ref{mass}) and Lemma \ref{lem21}, and since 
  $(\veps)_{\eps\in (0,1)}$ is bounded in $L^\infty(\Om\times (0,\infty))$ and in $L^\infty((0,\infty);W^{1,2}(\Om))$
  according to (\ref{vinfty}) and Lemma \ref{lem11}, it is clear that actually \eqref{u-weak-1plus2overn} and \eqref{24.5b} hold and $u$ and $v$ have all the regularity features 
  in (\ref{25.01}), and that the Vitali convergence theorem along with (\ref{F0}) and (\ref{F}) 
  ensures that
  \bas
	\ueps\to u,
	\quad \Feps(\ueps)\to u
	\quad \mbox{and} \quad
	\ueps\Feps'(\ueps) \to u
	\qquad \mbox{in } \bigcap_{p\in [1,\frac{n+2}{n})} L^p_{loc}(\bom\times [0,\infty))
  \eas
  as well as
  \bas
	\veps\to v
	\qquad \mbox{in } \bigcap_{p\in [1,\infty)} L^p_{loc}(\bom\times [0,\infty))
  \eas
  as $\eps=\eps_j\searrow 0$.
  Besides especially completing the verification of (\ref{24.2}) and (\ref{24.4}), and hence especially also of all the 
  regularity requirements in Definition \ref{dw}, together with %(\ref{24.3}) and
  (\ref{24.5})
  these two latter properties guarantee that
  \be{24.6}
	\ueps\Feps'(\ueps)\na\veps \wto u\na v
	\qquad \mbox{in } L^1_{loc}(\bom\times [0,\infty))
  \ee
  and  
  \be{24.7}
	\Feps(\ueps)\veps \to uv
	\qquad \mbox{in } L^1_{loc}(\bom\times [0,\infty))
  \ee
  as $\eps=\eps_j\searrow 0$, because
  \bas
	\frac{1}{\lim_{p\nearrow \frac{n+2}{n}} p} + \frac{1}{4} = \frac{5n+2}{4(n+2)} = 1-\frac{6-n}{4(n+2)} <1
  \eas
  and
  \bas
	\frac{1}{\lim_{p\nearrow \frac{n+2}{n}} p} + \frac{1}{\lim_{p\to \infty} p} = \frac{n}{n+2}<1.
  \eas
  Since for each $\eps\in (0,1)$ we have
  \bas
	-\int_0^\infty \io \ueps\vp_t - \io u^{(0)} \vp(\cdot,0)
	= - \int_0^\infty \io \na\ueps\cdot\na\vp
	+ \int_0^\infty \io \ueps\Feps'(\ueps)\na\veps\cdot\na\vp
  \eas
  for all $\vp\in C_0^\infty(\bom\times [0,\infty))$ and
  \bas
	- \int_0^\infty \io \veps \vp_t - \io v^{(0)} \vp(\cdot,0)
	= - \int_0^\infty \io \na\veps\cdot\na\vp
	- \int_0^\infty \io \Feps(\ueps)\veps\vp
  \eas
  for all $\vp\in C_0^\infty(\Om\times [0,\infty))$, taking $\eps=\eps_j\searrow 0$ we therefore readily obtain (\ref{wu})
  from (\ref{24.2}), (\ref{24.3}) and (\ref{24.6}), whereas (\ref{wv}) results from (\ref{24.4}), (\ref{24.5}) and (\ref{24.7}).
  Consequently, it follows that $(u,v)$ indeed forms a global weak solution of (\ref{0}) in the sense of Definition \ref{dw}.
\end{proof}
This essentially establishes our main result on global weak solvability in (\ref{0}) already.
\begin{proof}[Proof of Theorem \ref{theo25}]% \quad
  The statement on existence of a global weak solution fulfilling (\ref{25.01}) directly results from Lemma \ref{lem24}, 
  whereas the boundedness properties in (\ref{25.1}), (\ref{25.2}) and (\ref{25.3}) can readily be obtained upon combining (\ref{11.1}),
  (\ref{11.01}), (\ref{22.1}), (\ref{22.2}) and (\ref{11.4}) with (\ref{24.2}), \eqref{24.5b}, \eqref{u-weak-1plus2overn}, (\ref{24.3}) and (\ref{24.5}).
\end{proof}
\section{Stationary states}
We finally consider the stationary problem associated with (\ref{0}), that is, the boundary value problem
\begin{align}\label{stationary}
 0&=Δu-∇\cdot(u∇v)&&\text{ in } Ω\nn\\
 0&=Δv-uv&&\text{ in } Ω\\
 \frac{∂u}{∂ν}&=u\frac{∂v}{∂ν}, \quad v=\vst&&\text{ on }∂Ω,\nn
\end{align}
and our arguments in this regard will be closely 
related to those in \cite{braukhoff_lankeit}, where the second equation was instead supplemented by Robin-type boundary conditions. 
Dropping the requirement of radial symmetry here, we will assume that 
$Ω\subset ℝ^n$ is a bounded bounded with smooth boundary, and that with some $\beta\in (0,1)$, $\vst$ belongs to
$C^{2+β}(\Ombar)$ and is nonnegative on $\pO$.\abs
A first essential observation is that we can eliminate $u$ from the stationary system and subsequently deal with a single equation only. Especially regarding the question of uniqueness, the appearance of a constant parameter $\alpha$ in said equation could turn out to be unfortunate. However, we will later show that $\alpha$ is in one-to-one correspondence with $\io u$. (An alternative would be to compute $α=\frac{m}{\io e^v}$ and work with the nonlocal equation for $v$, see \cite{lee_wang_yang}, where this approach was used for the special case of constant boundary value.) 
\begin{lem}\label{lem:stationary-u}
 Let $v\in C^2(\Ombar)$. If $u\in C^2(\Ombar)$ satisfies 
 \begin{equation}\label{u-eq}
  \begin{cases}
   0=Δu-∇\cdot(u∇v)&\text{in } Ω\\
   \frac{∂u}{∂ν} = u \frac{∂v}{∂ν}&\text{ on }∂Ω
  \end{cases}
 \end{equation}
then there is $α\inℝ$ such that 
\begin{equation}\label{u-from-v}
 u=αe^v.
\end{equation}
If, on the other hand, \eqref{u-from-v} holds for some $v\in C^2(\Ombar)$ and $α\inℝ$, then \eqref{u-eq}. Furthermore, the signs of $\io u$ and $α$ coincide.
\end{lem}
\begin{proof}
 While the second part of the statement directly follows from the chain rule and the last part is obvious after integration of \eqref{u-from-v}, the first is identical to \cite[Lemma~4.1]{braukhoff_lankeit}.
 %(the boundary condition for $v$ does not affect the proof)
\end{proof}

We now take care of solvability and some a priori estimates for solutions of the second equation of \eqref{stationary} if we insert \eqref{u-from-v}, firstly in a related linear problem. 

\begin{lem}\label{lem:v-linear-regularity}
 Let $α\ge 0$, $\vst\in C^{2+β}(\Ombar)$, $\vst\ge 0$ and $v\in C^{β}(\Ombar)$. Then 
 \begin{equation}\label{v-linear}
  \begin{cases}
  Δ\tilde{v} = α\tilde{v} e^{v} &\text{in } Ω\\
  \tilde{v}=\vst&\text{on } ∂Ω
  \end{cases}
 \end{equation}
 has a unique solution $\tilde{v}\in C^{2+β}(\Ombar)$. This solution satisfies 
 \begin{equation}\label{v-linear-uniformbound}
  0\le \tilde{v} \le γ:= \max_{∂Ω} \vst.
 \end{equation}
Moreover, for every $α$ and $\vst$ as above, there is $C>0$ such that for every $v\in C^{β}(\Ombar)$ with $0\le v\le γ$ the corresponding solution $\tilde{v}$ satisfies 
\begin{equation}\label{v-holderbound}
\norm[C^{β}(\Ombar)]{v} \le C.
\end{equation}
\end{lem}
\begin{proof}
 Unique solvability results from \cite[Thm.~6.14]{GT}. If $α=0$, then \eqref{v-linear-uniformbound} follows from the classical maximum principle (\cite[Thm.~3.1]{GT}) for the harmonic function $\tilde{v}$ and the assumptions on $\vst$. For $α>0$, we note that if $\tilde{v}$ is minimal at some $x_0\in Ω$, then $0\le Δ\tilde{v}(x_0)=α\tilde{v}(x_0)e^{v(x_0)}$, due to the positivity of $αe^{v(x_0)}$ entails that $\tilde{v}\ge \min\tilde{v}=\tilde{v}(x_0)\ge 0$. If $\tilde{v}$, however, is minimal at some $x_0\in ∂Ω$, then, again, $\tilde{v}\ge \tilde{v}(x_0)=\vst(x_0)\ge 0$. With nonnegativity of $\tilde{v}$ thus ensured and hence $Δ\tilde{v}\ge 0$ in $Ω$, the second part of \eqref{v-linear-uniformbound} follows from the maximum principle (\cite[Cor.~3.2]{GT}). The Hölder bound in \eqref{v-holderbound} thereby can be concluded from the boundedness of the right-hand side in \eqref{v-linear} and elliptic regularity \cite[Thm.~8.29]{GT}.
\end{proof}

With this, solving the second equation of \eqref{stationary} with \eqref{u-from-v} is possible: 
\begin{lem}\label{lem:stationary-v}
 For every $α\ge 0$, the boundary value problem 
 \begin{equation}\label{stationary-v}
  \begin{cases}
   Δv = αve^v& \text{ in } Ω\\
   v=\vst&\text{ on } ∂Ω
  \end{cases}
 \end{equation}
has a solution $v\in C^2(\Ombar)$, and $v$ satisfies \eqref{v-linear-uniformbound} and \eqref{v-holderbound}.
\end{lem}
\begin{proof}
 With $c_1>0$ taken from \eqref{v-holderbound}, we introduce  
 \[
  X=\{w\in C^{β}(\Ombar)\mid 0\le w\le γ, \norm[C^{β}(\Ombar)]{w}\le c_1\}\subset C^{β}(\Ombar)
 \]
and for $v\in X$ let $Φ(v)=\tilde{v}$ denote the solution $\tilde{v}\in C^{2+β}(\Ombar)$ of \eqref{v-linear}. According to elliptic regularity theory (\cite[Thm.~6.6]{GT}) in conjunction with \eqref{v-holderbound}, $Φ(X)$ is bounded in $C^{2+β}(\Ombar)$, thus relatively compact in $C^{β}(\Ombar)$ and, again by Lemma~\ref{lem:v-linear-regularity}, $Φ(X)\subseteq X$. As $Φ\colon X\to X$ moreover is continuous, 
% v_n\to v in C^{β}(\Ombar) means that \norm[C^{2+β}(\Ombar)]{\tilde{v}_n} is bounded, therefore $\tilde{v_{n_k}}\to w in C^2 (along a subsequence). 
%It holds that Δw=αwe^v in Ω, w=\vst on ∂Ω; by the uniqueness statement in Lemma~\ref{lem:v-linear-regularity}, therefore w=Φ(v) and \tilde{v}_n\to w. 
Schauder's theorem asserts the existence of a fixed point $v=Φ(v)$.
\end{proof}

\begin{lem}\label{lem:v-monotone-wrt-alpha}
Let $v_1$ and $v_2$ be two solutions of \eqref{stationary-v} with $α_1\ge 0$ and $α_2\ge 0$, respectively.
 If $α_1\ge α_2$, then $v_1\le v_2$. 
\end{lem}
\begin{proof}
  Both $v_1$ and $v_2$ -- being solutions to \eqref{v-linear} with $v=v_1$ or $v=v_2$, respectively -- are nonnegative. We let $Ω_1:=\{x\inΩ\mid v_1(x)>v_2(x)\}$ and $\bar{v}=v_1-v_2$. As $x\mapsto xe^x$ is monotone for $x\in[0,\infty)$, 
 \begin{align*}
  Δ\bar{v} &= α_1v_1e^{v_1}-α_2v_2e^{v_2}\ge α_2(v_1e^{v_1}-v_2e^{v_2}) \ge 0 \qquad &&\text{in }Ω_1\\
  \bar{v}&=0\qquad&&\text{on }∂Ω_1.
 \end{align*}
By the maximum principle therefore $\max_{\Ombar}\bar{v} = \max_{∂Ω}\bar{v} = 0$ and thus $Ω_1=\emptyset$, so that $v_1\le v_2$ in $Ω\setminus Ω_1 = Ω$. 
\end{proof}
As a particular consequence of Lemma~\ref{lem:v-monotone-wrt-alpha}, for every $α\ge 0$ the solution to \eqref{stationary-v} is unique. From now on, we will denote it by $v_{α}$.

That, according to Lemma~\ref{lem:v-monotone-wrt-alpha}, $v_{α}$ is decreasing with respect to $α$ is of little help with regard to the monotonicity of $αe^{v_{α}}$ (or rather $α\io e^{v_{α}}$). For further information we study the derivative of $v_{α}$ w.r.t. $α$.

\begin{lem}\label{lem:vprime}
 For every $α_1> 0$, the function 
 \begin{equation}\label{vprime}
  v'_{α_1} = \frac{d}{dα} v_{α} \Big|_{α=α_1} = \lim_{α_2\to α_1} \frac{v_{α_2}-v_{α_1}}{α_2-α_1}
 \end{equation}
exists (with the limit taken in $C^2(\Ombar)$) and satisfies 
\begin{equation}\label{vprime-bvp}
 \begin{cases}
  Δv'_{α_1} = v_{α_1} + (α_1 e^{v_{α_1}}+α_1v_{α_1}e^{v_{α_1}})v'_{α_1}&\text{ in }Ω\\
  v'_{α_1}=0&\text{on } ∂Ω.
 \end{cases}
\end{equation}
\end{lem}
\begin{proof}
 For $α_1\ge 0$, $α_2\ge 0$, we let $w_{α_2,α_1}=\frac{v_{α_2}-v_{α_1}}{α_2-α_1}$ and note that $w=w_{α_2,α_1}$ solves 
 \begin{equation}\label{w}
 \begin{cases}
  Δw = f_{1,α_2}+f_{2,α_2,α_1} w &\text{ in } Ω\\
  w=0&\text{ on } ∂Ω
  \end{cases}
 \end{equation}
with $f_{1,α_2}=v_{α_2}e^{v_{α_2}}$, $f_{2,α_2,α_1}=α_1e^{v_{α_1}}+α_1v_{α_1}e^{v_{α_1}}F(v_{α_2}-v_{α_1})$, $F(z)=\frac{e^z-1}z$, $z\neq 0$, $F(0)=1$. 
From the Hölder bounds on $f_{1,α_2}$ and $f_{2,α_2,α_1}$ resulting from Lemma~\ref{lem:stationary-v} and \eqref{v-holderbound}, together with elliptic regularity theory (in the shape of \cite[Thm.~6.6]{GT}), we conclude that for every $A>0$ there is $C>0$ such that 
\[
 \norm[C^{2+β}(\Ombar)]{w_{α_2,α_1}}\le C \qquad \text{for all } α_1,α_2\in[0,A].
\]
If $(α_n)_{n\inℕ}\subset [0,∞)$ is a sequence with limit $α\in[0,∞)$, by Arzelà-Ascoli's theorem for every subsequence $(α_{n_k})_{k\inℕ}$ there are $w\in C^2(\Ombar)$ and $(α_{n_{k_l}})_{l\inℕ}$ such that $w_{α_{n_{k_l}}}\to w$ in $C^2(\Ombar)$ as $l\to\infty$ and, by \eqref{w}, $w$ satisfies 
\begin{equation}\label{w-limit}
 \begin{cases}
Δw= f_{1,α}+f_{2,α,α} w&\text{in } Ω\\
w=0&\text{on } ∂Ω.
 \end{cases}
\end{equation}
The solution to \eqref{w-limit} is unique (since $-\io |∇(w_1-w_2)|^2 =\io f_{2,α,α} (w_1-w_2)^2 \ge 0$ whenever $w_1$ and $w_2$ solve \eqref{w-limit}), hence actually $w_{α_n,α}\to w=v'$ in $C^2(\Ombar)$ as $n\to \infty$ and both existence of the limit in \eqref{vprime} and \eqref{vprime-bvp} follow.
\end{proof}

\begin{lem}\label{lem:vprime-estimates}
 For every $α>0$, 
 \[
  0 \ge v_{α}'>-\frac1{α}\qquad \text{in }Ω.
 \]
\end{lem}
\begin{proof}
We abbreviate $v'=v'_{α}$ and $v=v_{α}$. 
 From Lemma~\ref{lem:v-monotone-wrt-alpha}, we obtain that $0\ge v'$. We let $x_0\in \Ombar$ be such that $v'(x_0)=\min_{\Ombar}v'$. Then $v'\equiv 0$ (which would finish the proof) or $x_0\inΩ$ and due to \eqref{vprime-bvp} 
 \[
  0\le Δv' = ve^v + (αe^v+αve^v)v' \qquad \text{at } x_0,
 \]
so that by positivity of $e^{v(x_0)}$
\[
 0\le v+α(1+v)v'\qquad\text{at } x_0, 
\]
which yields 
\[
 v'(x_0)\ge -\frac1{α} \cdot \frac{v(x_0)}{1+v(x_0)}>-\frac1{α}.\qedhere
\]
\end{proof}
Consequences of Lemma~\ref{lem:vprime-estimates} on the desired relation between $α$ and $m=\io u$ are as follows: 
\begin{lem}\label{lem:m}
 The map
 \begin{equation}\label{def:m}
  m:\begin{cases}
     [0,∞)\to [0,∞)\\α\mapsto \io αe^{v_α}
    \end{cases}
 \end{equation}
is bijective.
\end{lem}
\begin{proof}
 Computing the derivative of $m$ (which uses Lemma~\ref{lem:vprime}), like in \cite[L.3.15]{braukhoff_lankeit}, we obtain 
 \[
  m'(α)=\io e^{v_{α}}+\io αe^{v_{α}}v'_{α} = \io e^{v_{α}}(1+αv'_{α}),
 \]
so that $m'(α)>0$ by Lemma~\ref{lem:vprime-estimates}, ensuring injectivity. Since $m(0)=0$ and $m(α)=\io αe^{v_{α}}\ge α|Ω|\to ∞$ as $α\to∞$, surjectivity is obvious.
\end{proof}

\begin{proof}[Proof of Theorem~\ref{th:stationary}]
  Combining Lemma~\ref{lem:stationary-u} with \eqref{def:m} shows that $(u,v)\in (C^2(\Ombar))^2$ solves \eqref{stationary} with 
  $\io u=m_0$ if and only if $m_0=m(α)$, $u=αe^v$ and $v=v_{α}$ solves \eqref{stationary-v}. 
  Bijectivity of $m$ (Lemma~\ref{lem:m}), existence and uniqueness of $v_{α}$ (Lemma~\ref{lem:stationary-v} 
  and Lemma~\ref{lem:v-monotone-wrt-alpha}) therefore imply the first part of Theorem~\ref{th:stationary}.\abs
  In the case when $\Om$ is a ball and $\vst$ is constant, radial symmetry follows from the above uniqueness statement.
  Since $u=α\exp(v)$ by Lemma~\ref{lem:stationary-u} and $\exp$ is monotone and convex, to complete the proof 
  it is sufficient to show convexity of $v$, that is of   
  the solution to \eqref{stationary-v}. But when written in radial coordinates, \eqref{stationary-v} turns into 
  \[
  	(r^{n-1}v_r)_r = αr^{n-1} ve^v,\qquad r\in(0,R),\qquad v(R)=\vst,
  \]
  with  $v_r(0)=0$ due to radial symmetry and differentiability of $v$. Hence, 
  \begin{equation}\label{vr}
 	v_r(r)=r^{1-n} \int_0^r αs^{n-1} v(s)e^{v(s)}ds = αr \int_0^1 t^{n-1}v(rt)e^{v(rt)}dt. 
  \end{equation}
  Nonnegativity of $v$ (cf. \eqref{v-linear-uniformbound}) shows that hence $v_r\ge 0$; thus the rightmost expression in \eqref{vr} is 
  clearly increasing with respect to $r$, which shows monotonicity of $v_r$ and therefore convexity of $v$.
\end{proof}

{\small
 \section*{Acknowledgement}
 The second author acknowledges support of the \textit{Deutsche Forschungsgemeinschaft} in the contex of the project \textit{Emergence of structures and advantages in cross-diffusion systems} (Project No. 411007140, GZ: WI 3707/5-1).
}
{
\small
\setlength{\parskip}{0pt}
\setlength{\itemsep}{0pt plus 0.3ex}
 %\bibliographystyle{abbrv}
 %\bibliography{lit.bib}
 
 \def\cprime{$'$}

}
\end{document}